\documentclass[11pt,a4paper]{article}

\usepackage[T1]{fontenc}
\usepackage{lmodern}
\usepackage[utf8]{inputenc}

\usepackage[left=2.45cm, top=2.45cm,bottom=2.45cm,right=2.45cm]{geometry}
\usepackage{amsmath}
\usepackage{amssymb,amsthm,graphicx,xcolor,tikz}

\usepackage[british]{babel}

\usepackage[
	bookmarks=true,
	bookmarksnumbered=true,
	bookmarksopen=true,
	unicode=true,
	pdftoolbar=true,
	pdfmenubar=true,
	pdffitwindow=false,
	pdfstartview={FitH},
	pdftitle={},
	pdfauthor={},
	pdfsubject={},
	pdfcreator={},
	pdfproducer={},
	pdfkeywords={},
	pdfnewwindow=true,
	colorlinks=true,
	linkcolor=black,
	citecolor=black,
	filecolor=black,
	urlcolor=black]{hyperref}

\makeatletter
	\def\@cite#1#2{[\textbf{#1}\if@tempswa , #2\fi]}	
	\def\@biblabel#1{[#1]}								
\makeatother

\numberwithin{equation}{section}
\numberwithin{figure}{section}

\newtheorem {theorem}{Theorem}[section]

\newtheorem {lemma}[theorem]{Lemma}
\newtheorem {corollary}[theorem]{Corollary}

\theoremstyle{definition}

\newtheorem {remark}[theorem]{Remark}

\newcommand{\BB}{\mathbb{B}}
\newcommand{\EE}{\mathbb{E}}
\newcommand{\HH}{\mathbb{H}}
\newcommand{\NN}{\mathbb{N}}
\newcommand{\PP}{\mathbb{P}}
\newcommand{\RR}{\mathbb{R}}
\renewcommand{\SS}{\mathbb{S}}
\newcommand{\XX}{\mathbb{X}}

\newcommand{\cF}{\mathcal{F}}
\newcommand{\cE}{\mathcal{E}}
\newcommand{\cH}{\mathcal{H}}
\newcommand{\cK}{\mathcal{K}}
\newcommand{\cL}{\mathcal{L}}
\newcommand\cW{\mathcal{W}}

\DeclareMathOperator{\Vol}{Vol}
\DeclareMathOperator{\pos}{pos}
\DeclareMathOperator{\Var}{Var}
\DeclareMathOperator{\Gr}{Gr}
\DeclareMathOperator{\Wass}{Wass}
\DeclareMathOperator{\Lip}{Lip}

\DeclareMathOperator{\interior}{int}
\DeclareMathOperator{\Bu}{B}
\DeclareMathOperator{\HT}{HT}
\DeclareMathOperator{\bd}{bd}

\newcommand{\dint}{\mathrm{d}}

\newcommand{\bbinom}[2]{\left[\!\!\begin{array}{c}#1\\#2\end{array}\!\!\right]}
\newcommand{\tbbinom}[2]{\scalebox{0.75}{$\left[\!\!\renewcommand{\arraystretch}{0.6}\begin{array}{c}#1\\#2\end{array}\!\!\right]$}}

\newcommand\tinyonehalf{\scalebox{0.65}{$\frac{1}{2}$}}

\newcommand{\overbar}[1]{\mkern 3.5mu\overline{\mkern-3.5mu#1\mkern-.0mu}\mkern .0mu}

\usepackage{accents}
\DeclareMathSymbol{\widetildesym}{\mathord}{largesymbols}{"65}
\newcommand{\lowerwidetildesym}{%
    \text{\smash{\raisebox{-1.3ex}{$\widetildesym$}}}
}
\newcommand{\overtilde}[1]{%
    \mathchoice
        {\accentset{\displaystyle\lowerwidetildesym}{#1}\mkern -4mu}
        {\accentset{\textstyle\lowerwidetildesym}{#1}\mkern -4mu}
        {\accentset{\scriptstyle\raisebox{-0.3ex}{$\lowerwidetildesym$}}{#1}\mkern -4mu}
        {\accentset{\scriptscriptstyle\raisebox{-0.3ex}{$\lowerwidetildesym$}}{#1}\mkern -4mu}
}

\begin{document}

\title{\bfseries Asymptotic normality for random polytopes\\ in non-Euclidean geometries}

\author{%
    Florian Besau\footnotemark[1]%
    \and Christoph Th\"ale\footnotemark[2]%
}

\date{}
\renewcommand{\thefootnote}{\fnsymbol{footnote}}
\footnotetext[1]{%
    Vienna University of Technology, Austria. Email: florian.besau@tuwien.ac.at
}

\footnotetext[2]{%
    Ruhr University Bochum, Germany. Email: christoph.thaele@rub.de
}

\maketitle

\begin{abstract}\noindent
    Asymptotic normality for the natural volume measure of random polytopes generated
    by random points distributed uniformly in a convex body in spherical or hyperbolic spaces is proved.
    Also the case of Hilbert geometries is treated and central limit theorems in Lutwak's dual Brunn--Minkowski theory are established.
    The results follow from a central limit theorem for weighted random polytopes in Euclidean spaces.
    In the background are Stein's method for normal approximation and
    geometric properties of weighted floating bodies.


\end{abstract}

\section{Introduction and main results}

\subsection{Motivation and background}

The study of random convex hulls is one of the core topics in stochastic geometry and has deep connections
to convex geometry and asymptotic geometric analysis;
we refer to the monographs \cite{IsotropicConvexBodies,SchneiderWeil} and the many references listed therein.
The most intensively investigated model can be described as follows.
Fix a compact convex set $K$, i.e., a convex body, in $\RR^d$ for some space dimension $d\geq 2$,
and assume that its volume (Lebesgue measure) $\Vol(K)$ is strictly positive.
Then, for $n\in\NN$, let $X_1,\ldots,X_n$ be independent random points sampled in $K$ according
to the uniform distribution $\Vol(\,\cdot\,|K)=\Vol(\,\cdot\,\cap K)/\Vol(K)$
(we shall adopt this notation for any measure that appears in this paper).
The convex hull of $X_1,\ldots,X_n$ is denoted by $K(n)=[X_1,\ldots,X_n]$.
Let us assume from now on that the boundary $\bd K$ of $K$ is sufficiently smooth in the sense that $\bd K$
is a twice differentiable $(d-1)$-submanifold of $\RR^d$ with Gauss--Kronecker curvature $H_{d-1}(x)>0$
for any $x\in\bd K$.
In this situation it is well known that the expected volume difference $\Vol(K)-\EE\Vol(K(n))$ satisfies
\begin{equation*}
    \Vol(K)-\EE\Vol(K(n)) = c_d\, (\Vol(K)/n)^{2/(d+1)} \left[\int_{\bd K}H_{d-1}(x)^{1/(d+1)}\,\cH^{d-1}(\dint x)\right](1+o_n(1)),
\end{equation*}
as $n\to \infty$, where $\cH^{d-1}$ denotes the $(d-1)$-dimensional Hausdorff measure and $c_d\in(0,\infty)$
is an explicitly known constant only depending on the space dimension $d$, see e.g.\ the survey article \cite{BaranySurvey}.
Let us emphasize that the curvature integral in the last formula is Blaschke's classical affine surface area of $K$,
a quantity, which was very intensively studied in the literature
\cite{HaberlParpatits:2014, Hug:1996, Leichtweiss:1986, LudwigReitzner:1999, LudwigReitzner:2010, Lutwak:1991, Petty:1985, SchuttWerner:1990}.
For the variance of $\Vol(K(n))$ it is known from \cite{ReitznerCLT2005} that
\begin{equation*}
    c\,n^{-\frac{d+3}{d+1}}\leq\Var\Vol(K(n)) \leq C\,n^{-\frac{d+3}{d+1}}
\end{equation*}
for all sufficiently large $n$ and where $c,C\in(0,\infty)$ are constants not depending on $n$.
Using Stein's method for dependency graphs Reitzner in his seminal paper \cite{ReitznerCLT2005}
has proven that the sequence of the suitably normalized random variables $\Vol(K(n))$ converges in distribution to a standard Gaussian random variable $Z$, i.e.,
\begin{equation*}
    \frac{\Vol(K(n))-\EE\Vol(K(n))}{\sqrt{\Var\Vol(K(n))}}\overset{d}\longrightarrow Z,
\end{equation*}
as $n\to\infty$, where we write $\overset{d}\longrightarrow$ to indicate convergence in distribution.

\medskip

The main goal of the present paper is to prove a similar central limit theorem for random polytopes in non-Euclidean geometries.
In particular, our focus lies on random convex hulls generated by uniformly distributed random points in a compact convex subset
of a homogeneous space of constant curvature $+1$ or $-1$.
In addition, we shall treat random convex hulls in Hilbert geometries based on a strictly convex set.
This continues a recent and very active line of research in stochastic geometry on non-Euclidean models, see e.g.\ \cite{BaranyHugReitznerSchneider, BHPS:2019, BrauchartEtAl,Calkaetc,DeussHoerrmannThaele,
HugReichenbacher,HugThaele,MaeharaMartini18,PenroseYukichMf}.
In addition, we are able to prove central limit theorems for dual volumes of random polytopes, which arise in Lutwak's dual Brunn--Minkowski theory.
Our approach combines two ingredients, namely Stein's method for normal approximation of functionals
of binomial point processes developed by Chatterjee \cite{Chatterjee} and Lachi\'eze-Rey and Peccati \cite{LachPecc} as well as
the concept of weighted floating bodies introduced by Werner \cite{Werner2002} and studied further by Besau, Ludwig and Werner \cite{BesauLudwigWerner}.
The so-called Malliavin--Stein technique, which is in the background of \cite{LachPecc}, was invented roughly 10 years ago and has led to a very large number of new and
deep limit theorems especially for models in stochastic geometry.
We refer the reader to the volume \cite{PeccatiReitznerBook}, which contains a representative collection of survey articles in this direction.
This technique was for the first time combined in \cite{Thaele18,ThaeleTurchiWespi,TurchiWespi} with geometric properties of classical floating bodies
to give quick and streamlined proofs of central limit theorems for various functionals of random polytopes in $\RR^d$.
In the present paper we develop this idea further by working with \emph{weighted} floating bodies and dealing with
\emph{weighted} volumes of random convex hulls in $\RR^d$.
In fact, it will turn out that all our results for non-Euclidean geometries can be deduced from our limit theorem in $\RR^d$ by choosing particular weight functions.

\medskip

The remaining parts of this text are structured as follows.
Our main results for random polytopes in spherical spaces are presented in Section \ref{subsec:CLTsphere},
those for hyperbolic spaces in Section \ref{subsec:CLThyperbolic} and
the central limit theorem in Hilbert geometries in Section \ref{subsec:CLThilbert}.
Furthermore, in Section \ref{subsec:CLTdualVolume} we also establish a limit theorem for the expectation and
a central limit theorem for the dual volumes in Lutwak's dual Brunn--Minkowski theory.
As explained above, all these results will follow from a central limit theorem for weighted random polytopes in Euclidean spaces,
which is presented in Section \ref{sec:CLTweightedEuclidean}.
The proof of this result is based on the Stein's method for normal approximation of functionals
of binomial point processes as well as on geometric properties of weighted floating bodies.
Some essential background material on these two topics is summarized in Section \ref{Background}.
All proofs are collected in Sections \ref{sec:ProofWeighted} and \ref{sec:ProofNonEuclidean} at the end of the paper.

\subsection{Central limit theorems in spherical spaces}\label{subsec:CLTsphere}

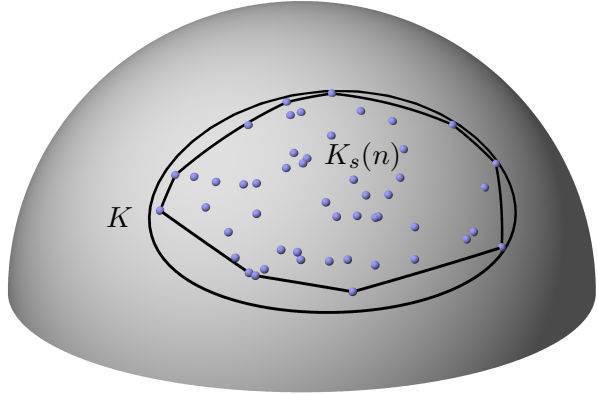
\begin{figure}[t]
    \centering
    \begin{tikzpicture}
        \clip (-4,-1.4) rectangle (4,4);
        \node at (0,0) {\includegraphics[width=0.8\textwidth]{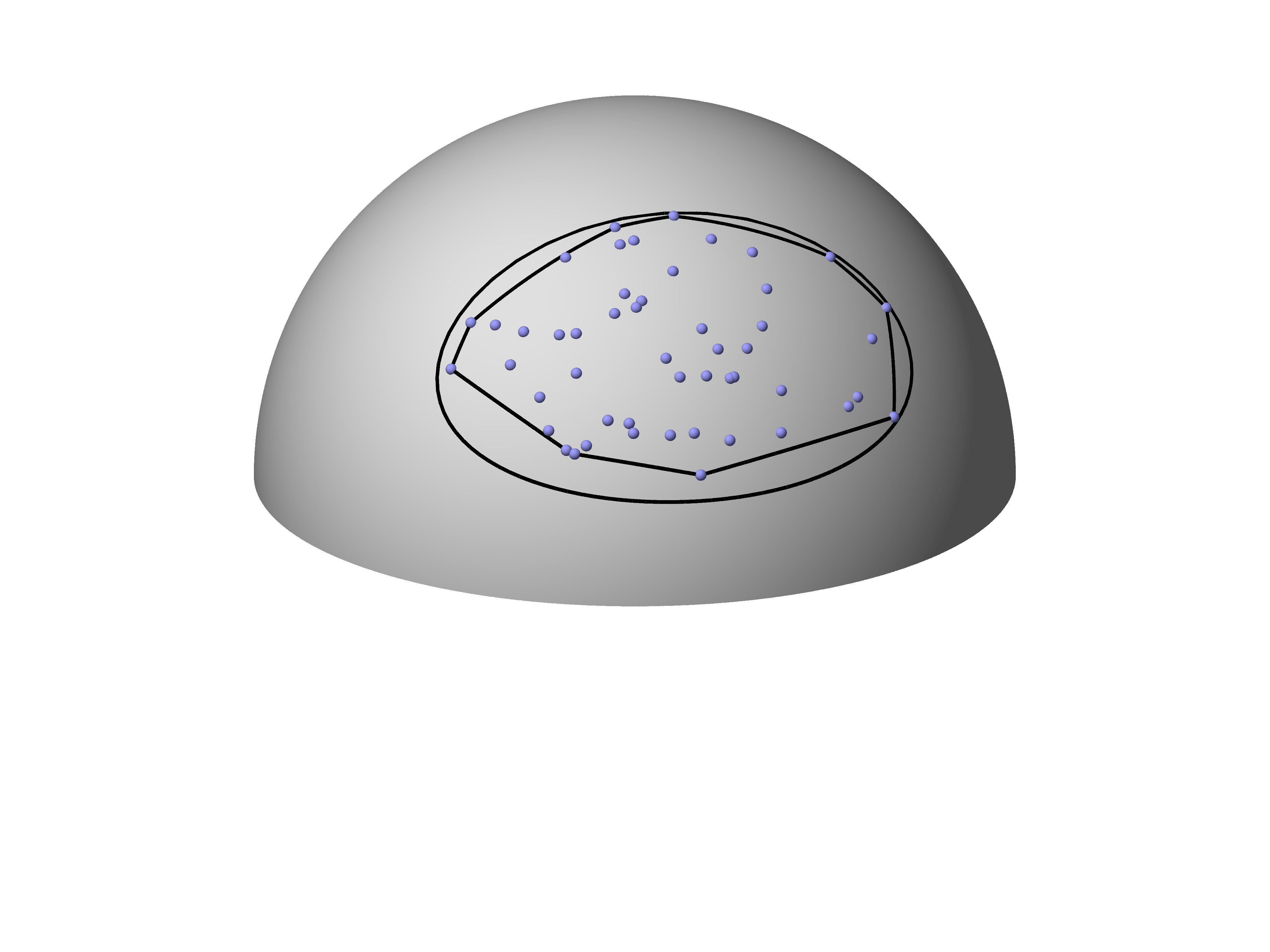}};
        \node at (-2.4,1) {$K$};
        \node at (0.8 ,1.8) {$K_s(n)$};
    \end{tikzpicture}
    \caption{Illustration of the spherical random polytope $K_s(n)$ generated in a spherical convex body $K$ contained in the open half-sphere $\SS^2_+$.}
    \label{fig:sphere}
\end{figure}

Let $d\geq 2$ and $\SS^{d}$ be the $d$-dimensional unit sphere in $\RR^{d+1}$.
A set $K\subset\SS^d$ is called spherically convex, provided that $K$ is contained in an open half-sphere and
if its positive hull
    $\pos K:=\{r x:x\in K,r\geq 0\}$
is a convex set in $\RR^{d+1}$ (what we call spherically
convex is called properly spherically convex by some authors).
By $\cK(\SS^d)$ we denote the space of all spherically convex sets.
Moreover, by $\cK_+^2(\SS^d)$ we denote the space of spherically convex sets
whose boundary is a twice differentiable $(d-1)$-submanifold of $\mathbb{S}^{d}$ and
such that the spherical Gauss--Kronecker curvature is strictly positive at any boundary point
(see, e.g., \cite[Section 4]{BesauWernerSpherical} for background material about spherical convex geometry).
For a set $B\subset\SS^d$, which is contained in an open half-sphere, we denote by
    $[B]_s:=[\pos B]\cap\SS^{d}$
its spherical convex hull.
Furthermore, $\Vol_s$ will denote the spherical Lebesgue measure on $\SS^{d}$.

The next theorem is the spherical analogue of the central limit theorem for the volume random polytopes in Euclidean spaces proved by Reitzner and
which was stated in the previous section.

\begin{theorem}\label{thm:Sphere}
    Let $K\in\cK_+^2(\SS^{d})$ and $X_1,X_2,\dotsc$ be a sequence of independent random points that are distributed in $K$ according to $\Vol_s(\,\cdot\,|K)$.
    For each $n\in\NN$ define $K_s(n):=[X_1,\dotsc,X_n]_s$. Then
    \begin{equation*}
        \frac{\Vol_s(K_s(n))-\EE\Vol_s(K_s(n))}{\sqrt{\Var\Vol_s(K_s(n))}}\overset{d}{\longrightarrow}Z,
    \end{equation*}
    as $n\to \infty$, where $Z$ is a standard Gaussian random variable.
\end{theorem}

We emphasize that the result of Theorem \ref{thm:Sphere} is in sharp contrast to the recent developments
\cite{BaranyHugReitznerSchneider,KabluchkoMarynychThaeleTemesvari} around random spherical convex hulls on half-spheres.
In fact, if in Theorem \ref{thm:Sphere} the set $K$ is a closed half-sphere, then the central limit theorem breaks down.
More precisely, if $X_1,X_2,\ldots$ is a sequence of independent random points distributed according
to the normalized spherical Lebesgue measure on the half-sphere
    $\SS^{d}_+:= \{x=(x_1,\ldots,x_{d+1})\in\RR^{d+1}:x_{d+1}\geq 0\}$
then, as $n\to\infty$,
\begin{equation}\label{eq:CLTcones}
    n\left(\frac{\Vol_s(\SS^{d})}{2}-\Vol_s(K_s(n))\right) \overset{d}{\longrightarrow} \int_{\RR^d\setminus[\Pi_d]}\frac{\dint x}{\|x\|^{d+1}},
\end{equation}
where $[\Pi_d]$ denotes the convex hull in $\RR^d$ of a Poisson point process $\Pi_d$ on $\RR^d$ whose intensity measure has density
$x\mapsto \frac{2}{\Vol_s(\SS^{d})}\frac{1}{\|x\|^{d+1}}$, $x\in\RR^d\setminus\{o\}$,
with respect to the Lebesgue measure, see \cite[Theorem 2.6]{KabluchkoMarynychThaeleTemesvari}.
Clearly, the limiting random variable on the right hand side in \eqref{eq:CLTcones} is non-Gaussian.

\subsection{Central limit theorems in hyperbolic spaces}\label{subsec:CLThyperbolic}

\begin{figure}[t]
    \centering
    \begin{tikzpicture}
        \clip (-5,-0.5) rectangle (5,4.5);
        \node at (0,0) {\includegraphics[width=0.8\textwidth]{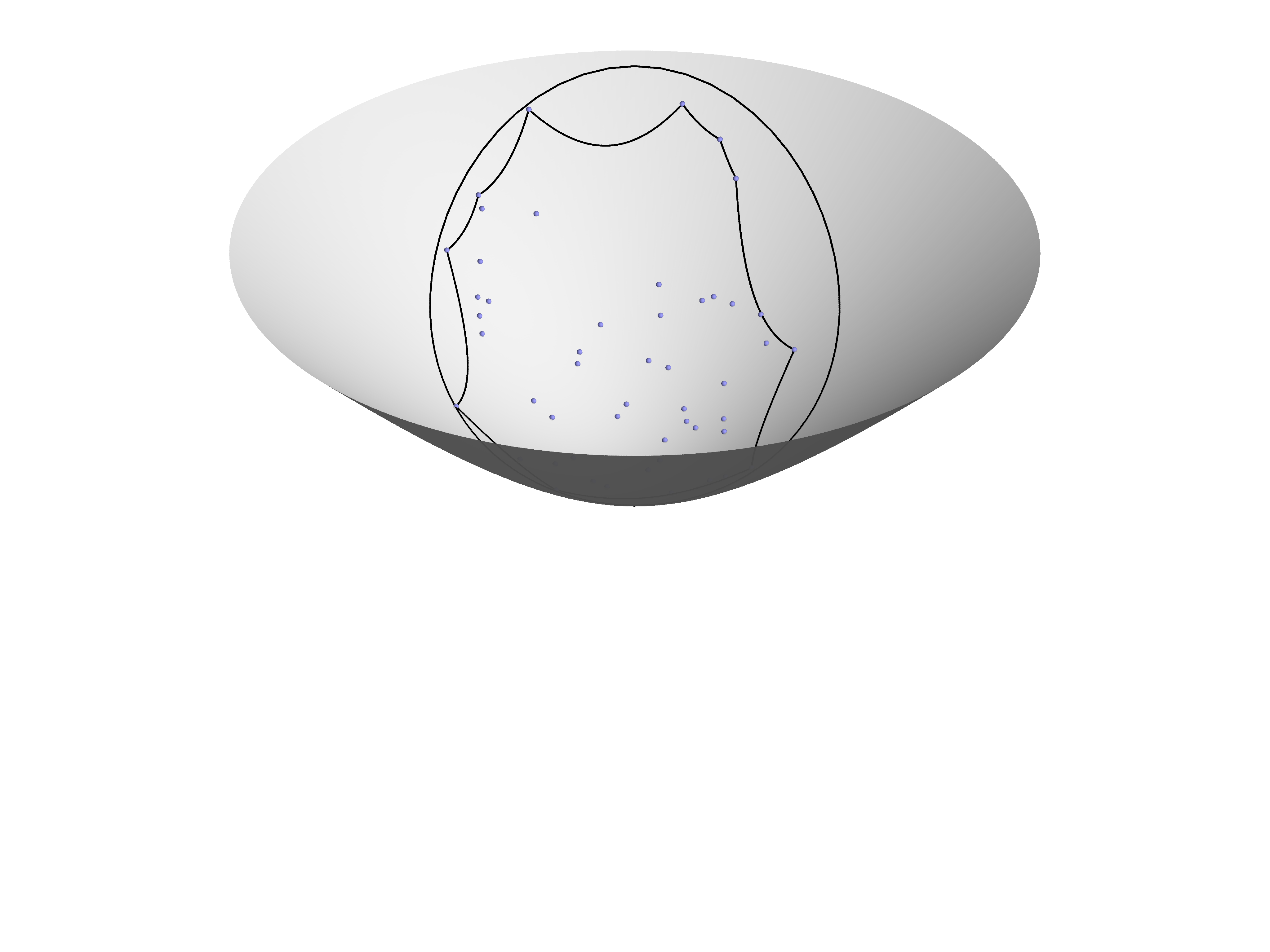}};
        \node at (2.5,2) {$K$};
        \node at (0 ,3) {$K_h(n)$};
    \end{tikzpicture}
    \caption{Illustration of the hyperbolic random polytope $K_h(n)$ generated in a hyperbolic convex body $K$
    in the hyperbolid model of the hyperbolic plane $\HH^2$.}
    \label{fig:hyperbolid}
\end{figure}
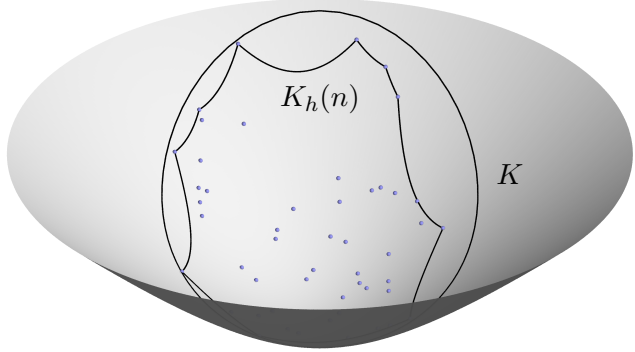

Having presented our result for spherical space, we turn now to the hyperbolic case.
We let $\RR^{d,1}$ for $d\geq 2$ be the $(d+1)$-dimensional Lorentz--Minkowski space,
by which we understand $\RR^{d+1}$ equipped with the indefinite inner product
    $x\circ x:=x_1^2+\ldots+x_d^2-x_{d+1}^2$, $x=(x_1,\ldots,x_{d+1})\in\RR^{d+1}$.
Our model for the hyperbolic space is the hyperboloid
    $\HH^d:=\{x\in\RR^{d,1}:x\circ x=-1,x_{d+1}>0\}$.
Similarly to the spherical set-up, a set $K\subset\HH^d$ is called hyperbolically convex if it is compact and if $\pos K$ is convex in $\RR^{d+1}$.
We let $\cK(\HH^d)$ be the space of hyperbolically convex sets and by $\cK_+^2(\HH^d)$
we denote the subspace of hyperbolically convex sets with the property that the hyperbolic Gauss--Kronecker curvature is strictly positive at every boundary point
(see, e.g., \cite[Section 3]{BesauWernerSpaceForms} for background material on hyperbolically convex sets).
For a set $B\subset\HH^d$ we denote by
    $[B]_h:=[\pos B]\cap\HH^d$
the hyperbolic convex hull of $B$.
Finally, by $\Vol_h$ we denote the natural hyperbolic volume measure on $\HH^d$, i.e.,
the $d$-dimensional Hausdorff measure on $\HH^d$ induced by the hyperbolic distance.

Our next theorem is the hyperbolic analogue of Theorem \ref{thm:Sphere}.

\begin{theorem}\label{thm:Hyperbolic}
    Let $K\in\cK_+^2(\HH^{d})$ and $X_1,X_2,\dotsc$ be a sequence of independent random points that are distributed in $K$
    according to $\Vol_h(\,\cdot\,|K)$.
    For each $n\in\NN$ define $K_h(n):=[X_1,\dotsc,X_n]_h$. Then
    \begin{equation*}
        \frac{\Vol_h(K_h(n))-\EE\Vol_h(K_h(n))}{\sqrt{\Var\Vol_h(K_h(n))}}\overset{d}{\longrightarrow}Z
    \end{equation*}
    as $n\to\infty$, where $Z$ is a standard Gaussian random variable.
\end{theorem}

\begin{remark}
    Of course Theorem \ref{thm:Hyperbolic} holds true in all other equivalent models of $d$-dimensional hyperbolic space,
    such as, for example, the Poincar\'e (or conformal ball) model inside the $d$-dimensional unit ball $B_2^d$ or the upper half-space model.
    Figure \ref{fig:poincare_hilbert} shows an illustration of a hyperbolic random polytope in a hyperbolic disc in the Poincar\'e model for the hyperbolic plane $\HH^2$.
    Another model is the projective model for $\HH^d$ inside $B_2^d$, which is discussed also in Remark \ref{rem:projectivemodel} below.
\end{remark}

\subsection{Central limit theorems in Hilbert geometries}\label{subsec:CLThilbert}

Fix a compact convex subset $C\subset\RR^d$, $d\geq 2$, with non-empty interior, i.e., $\interior C\neq\emptyset$.
For two distinct points $x,y\in\interior C$, the line through $x$ and $y$ intersects the boundary of $C$ in
precisely two points $p=p(x,y)$ and $q=q(x,y)$ so that one has the points $p,x,y,q$ on the line in that order.
The Hilbert distance between $x$ and $y$ is defined via the cross-ratio of these four points by
\begin{equation*}
    d_C(x,y) := \frac{1}{2}\log\bigg(\frac{\|y-p\|}{\|x-p\|} \frac{\|x-q\|}{\|y-q\|}\bigg),
\end{equation*}
where $\|x\|=\sqrt{x\cdot x}$ stands for the Euclidean norm of $x\in\RR^d$.
The pair $(C,d_C)$ is known as a Hilbert geometry. Notice that if $C=B_2^d$, then $(C,d_C)$ is the classical projective model of the $d$-dimensional hyperbolic space.
Furthermore, since the cross-ratio is invariant with respect to projective transformations, we immediately see that any projective transformation $\Phi$ yields an isometry between $(C,d_C)$ and $(\Phi C,d_{\Phi C})$. Hence, if $C$ is an ellipsoid, then the Hilbert geometry determined by $C$ is isometric to the hyperbolic space $\HH^d$ considered in the previous section.
Hilbert geometries are important examples of Finsler manifolds, i.e., differentiable manifolds with a Finsler metric on the tangent bundle, which are generalizations of Riemannian manifolds. In particular, if $(C,d_C)$ carries a Riemannian structure, then $C$ has to be an ellipsoid, in other words, the only Riemmanian Hilbert geometry is the hyperbolic space, see e.g.\ \cite[Theorem 11.6]{Troyanov:2014}. We further refer to the handbook \cite{PapaTroyanov} for a representative overview on the topic of Hilbert geometries.

In what follows we shall assume that $C$ is strictly convex, since in this case affine hyperplanes are the only totally geodesic submanifolds of dimension $d-1$.
By $\cK_+^2(C)$ we denote the space of convex subsets $K\subset\interior C$ whose boundary $\bd K$
is a twice differentiable submanifold of $\RR^d$ with strictly positive Gauss--Kronecker curvature in each boundary point.
There are several reasonable choices for a volume measure in a Hilbert geometry.
We restrict our attention to two prominent examples.
The first is the Busemann volume $\Vol_C^{\Bu}$, which is the $d$-dimensional Hausdorff measure on the metric space $(C,d_C)$.
The other one is the Holmes--Thompson volume $\Vol_C^{\HT}$, a notion that is closely related to the symplectic structure on $\RR^{2d}$, see \cite{AlvarezThompson,PapaTroyanov,ThompsonBook} and also Remark \ref{rem:HilbertVolumes} at the end of the paper.

Our next result is a central limit theorem for random polytopes in Hilbert geometries.
\begin{theorem}\label{thm:HilbertGeometries}
    Fix a strictly convex compact set $C\subset\RR^d$, let $K\in\cK_+^2(C)$ and $\diamondsuit\in\{{\Bu},{\HT}\}$.
    Let $X_1,X_2,\dotsc$ be a sequence of independent random points that are distributed in $K$ according to $\Vol_C^\diamondsuit(\,\cdot\,|K)$.
    For each $n\in\NN$ let $K_C(n):=[X_1,\dotsc,X_n]$.
    Then
    \begin{equation*}
        \frac{\Vol_C^\diamondsuit(K_C(n))-\EE\Vol_C^\diamondsuit(K_C(n))}{\sqrt{\Var\Vol_C^\diamondsuit(K_C(n))}}\overset{d}{\longrightarrow}Z
    \end{equation*}
    as $n\to\infty$, where $Z$ is a standard Gaussian random variable.
\end{theorem}

\begin{remark}
    Theorem \ref{thm:HilbertGeometries} can be seen as an extension of Theorem \ref{thm:Hyperbolic}, since, as mentioned above, if $C$ is an ellipsoid then $(C,d_C)$ is isometric to $\HH^d$. Furthermore, in this case the hyperbolic group of motions acts invariant on $(C,d_C)$ and therefore any natural definition of volume on $(C,d_C)$ agrees, up to a positive constant, with the natural Lebesgue measure on $\HH^d$ that is invariant with respect to the hyperbolic group of motions.
\end{remark}

\begin{figure}[t]
    \centering
    \begin{tikzpicture}
        \begin{scope}
        \clip (-3.5,-3.5) rectangle (3.5,3.5);
        \node at (0,0) {\includegraphics[width=0.45\textwidth]{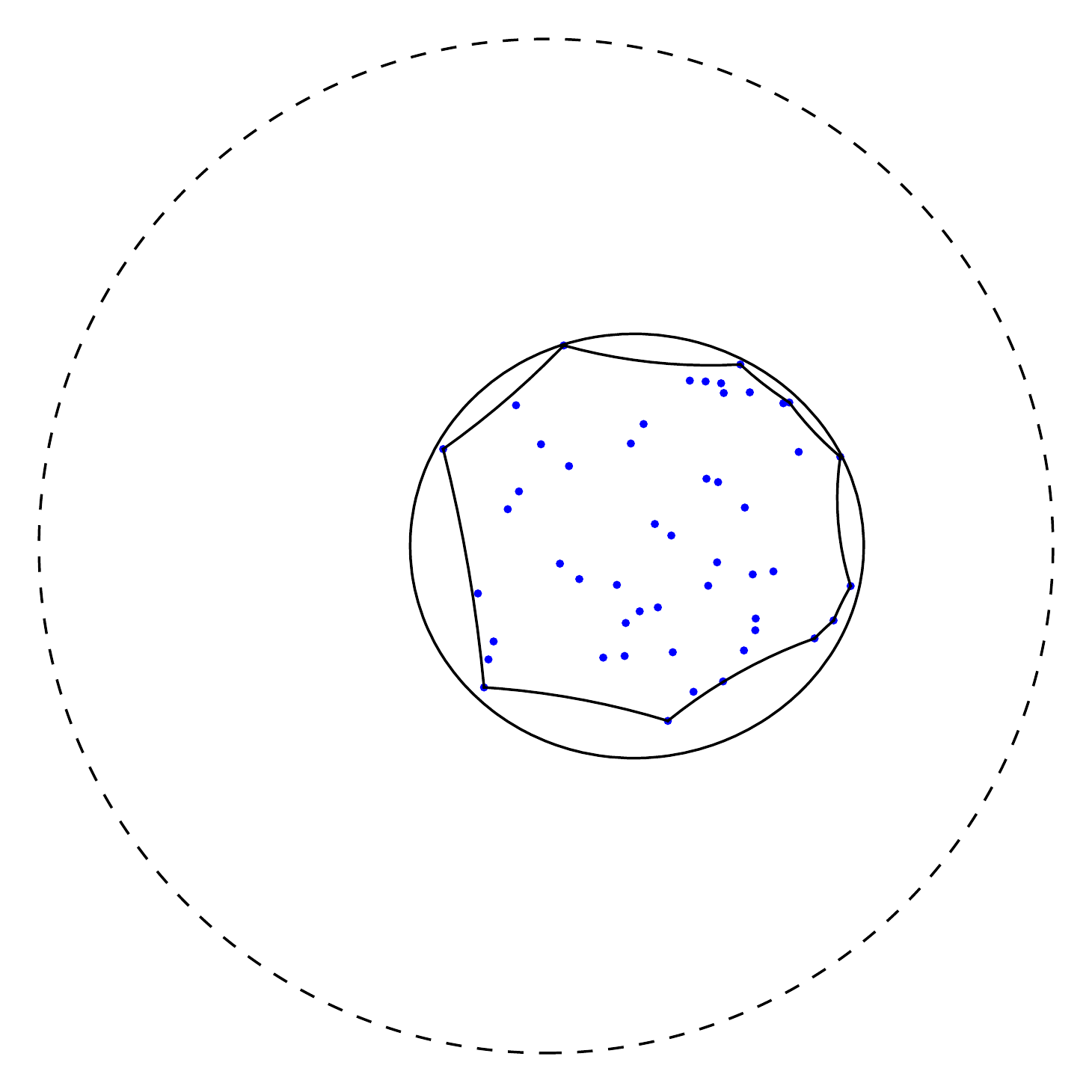}};
        \node at (-1.1,-.4) {\small $K$};
        \node at (0.4,1) {\small $K_h(n)$};
        \end{scope}

        \begin{scope}[xshift=.5\textwidth]
        \clip (-3.5,-3.5) rectangle (3.5,3.5);
        \node[rotate=40] at (0, 0) {\includegraphics[width=0.55\textwidth]{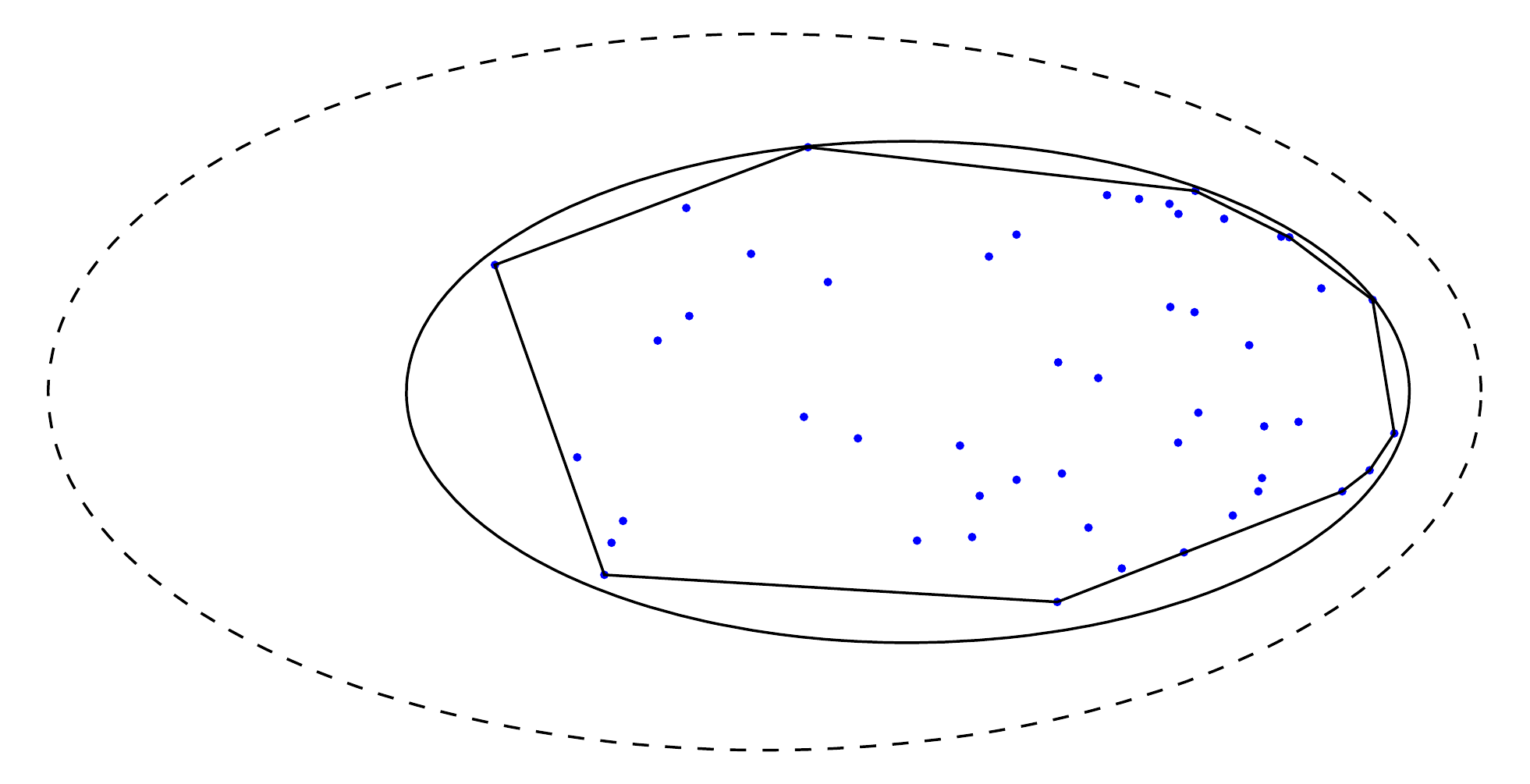}};
        \node at (-2, -1) {$K$};
        \node at (0.2, -0.5) {$K_C(n)$};
        \node at (-1.6,1.6) {$C$};
        \end{scope}
    \end{tikzpicture}
    \caption{Left: Illustration of the hyperbolic random polytope $K_h(n)$ in the Poincar\'e model for $\HH^2$.
    Right: Illustration of a random polytope $K_C(n)$ in a convex body $K$ in a Hilbert geometry with $C$ being an ellipse.}
    \label{fig:poincare_hilbert}
\end{figure}
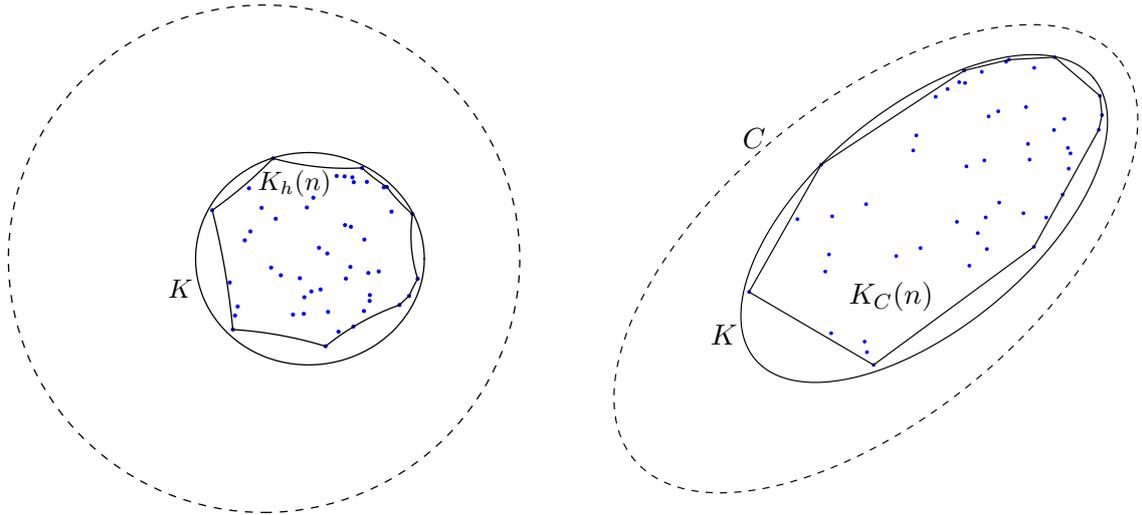

\subsection{Central limit theorems in the dual Brunn--Minkowski theory}\label{subsec:CLTdualVolume}

We denote by $\cK_+^2(\RR^d)$ the space of convex bodies $K\subset\RR^d$ with twice differentiable boundary having strictly positive Gauss--Kronecker curvature $H_{d-1}(x)>0$ at every boundary point $x\in\bd K$. A central limit theorem for the intrinsic volumes $V_j$ of the random polytopes $K(n)=[X_1,\ldots,X_n]$ generated as the (Euclidean) convex hull of $n\geq d+1$ independent and uniformly distributed random points $X_1,\ldots,X_n$ in convex bodies $K\in\cK_+^2(\RR^d)$ has recently been established in \cite{ThaeleTurchiWespi}, as $n\to\infty$. We recall that the intrinsic volumes $V_j(K)$ may be defined by Kubota's formula as the mean volume of the $j$-dimensional orthogonal projections of $K$, that is,
\begin{equation*}
    V_j(K) = \bbinom{d}{j} \EE \Vol(\pi_E K),
\end{equation*}
where $E$ is a uniformly distributed random linear subspace in the Grassmannian $\Gr_j(\RR^d)$ of all $j$-dimensional linear subspaces of $\RR^d$ and $\pi_E$ denotes the orthogonal projection onto $E\in \Gr_j(\RR^d)$. The dimensional constant $\tbbinom{d}{j}$ is the ball-binomial, i.e.,
\begin{equation*}
	\bbinom{d}{j} := \binom{d}{j} \frac{\Vol(B_2^d)}{\Vol(B_2^j)\Vol(B_2^{d-j})} = \frac{1}{2} \frac{B(\frac{j}{2}, \frac{d-j}{2})}{B(j,d-j)},
\end{equation*}
where $B_2^d$ is the $d$-dimensional Euclidean unit ball and $B(x,y)$, $x,y>0$, is the Beta function.

\medskip
Lutwak \cite{Lutwak:1975,Lutwak:1979} introduced the notion of \emph{dual mixed volumes} in the 1970s, which marked the beginning of a rapidly developing theory that has grown to be a major research topic and is now often referred to as the \emph{dual Brunn--Minkowski theory}, see, for example, \cite{AHH:2018, Bernig:2014, BHP:2018, HLYZ:2016, LYZ:2018}. Here, by ``dual'' one does not in general refer to a strict duality, but rather an informal impression that has been drawn by many researchers who unveiled results for dual mixed volumes or notations derived from them that seem to mirror classical theorems from the Brunn--Minkowski theory for mixed volumes. Hence, the connection between the dual Brunn--Minkowski theory and the classical Brunn--Minkowski theory are often only by name.

For a convex body $K\subset \RR^d$ that contains the origin $o$ in its interior $\interior K$, the \emph{$j$th-dual volume} $\overtilde{V}_j(K)$ is defined as the mean $j$-dimensional intersection volume, that is,
\begin{equation*}
	\overtilde{V}_j(K) = \bbinom{d}{j} \EE \Vol(K\cap E),
\end{equation*}
where $E\in\Gr_j(\RR^d)$ is a $j$-dimensional linear subspace of $\RR^d$ distributed according to the rotation invariant Haar probability measure on $\Gr_j(\RR^d)$.
The dual volumes are normalized such that $\overtilde{V}_j(B_2^d) = V_j(B_2^d) = \binom{d}{j}\Vol(B_2^d) / \Vol(B_2^{d-j})$ and $V_n = \overtilde{V}_n = \Vol$.
Note that the dual volumes are also known as \emph{dual quermassintegrals} in the literature and they can be extended to bounded Borel sets, see \cite{Gardner:2007, GJV:2003}. We remark that the dual volumes have already appeared in stochastic geometry in connection with the expected f-vector of a class of Poisson polyhedra \cite{HHRT}.

A limit theorem for the expected intrinsic volumes of $K(n)$ has been established by B\'ar\'any \cite{Barany:1992} for convex bodies of class $C^3_+$ and was extended in \cite{BoroczkyHoffmanHug:2008,Reitzner:2004}. Combing the most complete version \cite[Theorem 1.1]{BoroczkyHoffmanHug:2008}
with the calculations for the unit ball by Affentranger \cite[Theorem 2]{Affentranger:1991} one finds, for $j\in\{1,\dotsc,d\}$, that
\begin{equation}\label{eqn:limit_V_j}
    \lim_{n\to\infty} \left(\frac{n}{\Vol(K)}\right)^{\frac{2}{d+1}} \Big( V_j(K) - \EE V_j(K(n))\Big)
    = c(d,j) \int_{\bd K} H_{d-1}(x)^{\frac{1}{d+1}}\, H_{d-j}(x)\, \cH^{d-1}(\dint x),
\end{equation}
where $K$ is a convex body that admits a rolling ball from the inside,
$H_{j}(x)$ is the normalized elementary symmetric function of the (generalized) principal curvatures of $\bd K$ at $x$ and
\begin{equation}
    c(d,j)
    = \frac{1}{2\Vol(B_2^{d-j})} \binom{d-1}{j-1}\ \frac{d+1}{d+3}
        \ \frac{1}{j!} \Gamma\left(j+\frac{d+3}{d+1}\right)
        \left(\frac{d+1}{\Vol(B_2^{d-1})}\right)^{\!\!\frac{2}{d+1}}.
\end{equation}

The following theorem is ``dual'' to \eqref{eqn:limit_V_j} and follows from the weighted limit theorem established in \cite[Theorem 3.1]{BoroczkyFodorHug2010}, see \eqref{eqn:BFH_Thm} below.

\begin{theorem}\label{thm:dualVolumeExp}
    Let $j\in \{1,\dotsc, d-1\}$ and let $K\in \cK_+^2(\RR^d)$ such that $o\in\interior K$. Let further $X_1, X_2,\dotsc$ be a sequence of independent random points that are distributed in $K$ according to $\Vol(\, \cdot \,|K)$. For each $n\in \NN$ define $K(n):=[X_1,\dotsc,X_n]$. Then
    \begin{equation*}
        \lim_{n\to \infty} \left(\frac{n}{\Vol(K)}\right)^{\frac{2}{d+1}} \Big(\overtilde{V}_j(K) - \EE \overtilde{V}_j(K(n))\Big)
        = \tilde{c}(d,j) \int_{\bd K} H_{d-1}(x)^{\frac{1}{d+1}}\, \|x\|^{j-d} \,\cH^{d-1}(\dint x),
    \end{equation*}
    where
    \begin{equation}\label{eqn:const}
        \tilde{c}(d,j)
        = \frac{1}{2 \Vol(B_2^{d-j})} \binom{d-1}{j-1} \ \frac{d+1}{d+3}
            \ \frac{1}{d!} \Gamma\left(d+ \frac{d+3}{d+1}\right)
            \left(\frac{d+1}{\Vol(B_2^{d-1})}\right)^{\!\!\frac{2}{d+1}}.
    \end{equation}
\end{theorem}

Finally, we also establish a central limit theorem for the dual volumes of the random polytopes $K(n)$.

\begin{theorem}\label{thm:dualVolume}
    Let $j\in\{1,\dotsc,d\}$, $K\in \cK_+^2(\RR^d)$ and assume that  $o\in\interior K$.
    Let further $X_1,X_2,\dotsc$ be a sequence of independent random points that are distributed in $K$ according to $\Vol(\,\cdot\,|K)$.
    For each $n\in \NN$ define $K(n) := [X_1,\dotsc,X_n]$.
    Then
    \begin{equation*}
        \frac{\overtilde{V}_j(K(n))-\EE\overtilde{V}_j(K(n))}{\sqrt{\Var\overtilde{V}_j(K(n))}}\overset{d}{\longrightarrow}Z
    \end{equation*}
    as $n\to \infty$, where $Z$ is a standard Gaussian random variable.
\end{theorem}


\section{Central limit theorem for weighted Euclidean spaces}\label{sec:CLTweightedEuclidean}

Fix a space dimension $d\geq 2$ and recall that a convex body $K\subset\RR^d$ is a compact, convex subset with non-empty interior, i.e., $\interior K \neq \emptyset$. By $\cK(\RR^d)$ we denote the space of convex bodies in $\RR^d$. In addition, we let $\cK_+^2(\RR^d)$ be the space of convex bodies $K$ in $\RR^d$ such that the boundary $\bd K$ is a twice-differentiable $(d-1)$-submanifold of $\RR^d$ with strictly positive Gaussian curvature $H_{d-1}(x)$ at every boundary point $x\in\bd K$. For $K\in\cK(\RR^d)$ we let $\cW(K)$ be the class of weight functions on $K$, by which we mean the set of measurable functions $\varphi:K\to (0, \infty]$ with the following properties:
\begin{enumerate}
    \item[a)] (probability density): $\int_K\varphi(x)\,\dint x = 1$,
    \item[b)] (absolute lower bound): there is a constant $c_{\varphi}>0$ such that $c_{\varphi} \leq \varphi(x)$ for all $x\in K$.
    \item[c)] (continuous and bounded around the boundary): there exists an convex body $L$ in the interior of $K$ such that $\varphi$ is continuous on $K\setminus L$ and
        $\varphi(x) \leq C_{\varphi}$ for all $x\in K\setminus L$ for some positive constant $C_\varphi>0$.
\end{enumerate}
Clearly, if $\varphi:K\to (0, \infty)$ is continuous on $K$ and $\int_K \varphi(x)\,\dint x =1$, then $\varphi \in \cW(K)$.
Each $\varphi\in\cW(K)$ can be regarded as the density function of a probability measure $\Phi$ on $K$, that is, $\Phi(B)=\int_B\varphi(x)\,\dint x$ for all measurable subsets $B\subset K$.
As in the previous sections, we denote by $[B]$ the (Euclidean) convex hull of a set $B\subset\mathbb{R}^d$.

For $K\in\cK(\RR^d)$, $\varphi\in\cW(K)$ and $n\geq d+1$ let $X_1,\ldots,X_n$ be independent random points with distribution $\Phi$. The convex hull
\begin{equation}\label{eq:DefRandomPolytopes}
    K_\varphi(n):=[X_1,\ldots,X_n]
\end{equation}
of these points is a weighted random polytope contained in $K$. If $\varphi = \Vol(K)^{-1}$ then $\Phi$ is the uniform distribution on $K$ and $K_\varphi(n)$ reduces to the uniform model for random polytopes, which was intensively studied in the literature and which we also discussed in the Introduction. A quantity of particular interest is the volume $\Vol(K_\varphi(n))$ of $K_\varphi(n)$ and its asymptotic behaviour, as $n\to\infty$. More generally, for another weight function $\psi\in\cW(K)$ we investigate the weighted volume $\Psi(K_\varphi(n))=\int_{K_\varphi(n)}\psi(x)\,\dint x$ of $K_\varphi(n)$. The asymptotic behaviour of the expectation $\EE\Psi(K_\varphi(n))$ was studied in \cite{BoroczkyFodorHug2010}. In particular, \cite[Theorem 3.1]{BoroczkyFodorHug2010} shows that
\begin{equation}\label{eqn:BFH_Thm}
	\lim_{n\to\infty}n^\frac{2}{d+1}(1-\EE\Psi(K_\varphi(n))) = c(d,d)\int_{\bd K}\varphi(x)^{-\frac{2}{d+1}}\,H_{d-1}(x)^\frac{1}{d+1}\,\psi(x)\,\cH^{d-1}(\dint x)
\end{equation}
is valid for $K\in\cK_+^2(\RR^d)$ and $\varphi,\psi\in\cW(K)$, where $c(d,d)$ is the constant \eqref{eqn:const}.

\medskip

In this paper we shall prove that the suitably centred and normalized weighted volumes $\Psi(K_\varphi(n))$ satisfy a central limit theorem. In fact, the next theorem can be regarded as our main contribution and Theorems \ref{thm:Sphere}, \ref{thm:Hyperbolic}, \ref{thm:HilbertGeometries} and \ref{thm:dualVolume} presented in the previous sections will all follow from this result. At the same time it generalizes Reitzner's central limit theorem from \cite{ReitznerCLT2005} for the ordinary volume to weighted volumes and also to arbitrary underlying densities. We emphasize that a major obstacle in the proof of such a result is a lower variance bound, which will separately be provided in Theorem \ref{thm:LowerVarianceBound} below.

\begin{theorem}\label{thm:WeightedVolumeEuclidean}
    Let $K\in\cK_+^2(\RR^d)$ and $\varphi,\psi\in\cW(K)$ and let $X_1,X_2,\ldots$ be independent random points with distribution
    $\Phi$ and define $K_\varphi(n):=[X_1,\ldots,X_n]$ for each $n\in\NN$. Then
    \begin{equation*}
        \frac{\Psi(K_\varphi(n))-\EE\Psi(K_\varphi(n))}{\sqrt{\Var\Psi(K_\varphi(n))}} \overset{d}{\longrightarrow} Z
    \end{equation*}
    as $n\to\infty$, where $Z$ is a standard Gaussian random variable.
\end{theorem}

Our proof of Theorem \ref{thm:WeightedVolumeEuclidean} is based on two principal ingredients, one of geometric and the other of probabilistic nature.
The first is the concept of weighted floating bodies, which was introduced in \cite{Werner2002} and very recently further studied in \cite{BesauLudwigWerner},
and the fact from \cite{VuConcentration} that a suitable weighted floating body is contained in $K_\varphi(n)$ with high probability.
The other ingredient is a version of Stein's method, which is a powerful probabilistic device to prove central limit theorems.
Here, we use a version for functionals of binomial point processes from \cite{LachPecc}, which extends the earlier ideas developed in \cite{Chatterjee}.
A similar approach was also used in the recent papers \cite{Thaele18,ThaeleTurchiWespi,TurchiWespi}, where asymptotic normality for
intrinsic volumes of non-weighted random polytopes in Euclidean spaces was studied.
In the present paper we develop this technique further and combine it with geometric properties of \textit{weighted} floating bodies to
make it work for \textit{weighted} volumes of \textit{weighted} random polytopes as well.

\begin{remark}
    In addition to what has been presented so far we will actually prove the following quantitative version of Theorem \ref{thm:WeightedVolumeEuclidean}.
    The Wasserstein distance (see \eqref{eq:DefWasserstein} below) between the law of
        $\Big(\Psi(K_\varphi(n))-\EE\Psi(K_\varphi(n))\Big)/\sqrt{\Var\Psi(K_\varphi(n))}$
    and that of the standard Gaussian random variable $Z$ is bounded by a constant multiple of
        $(\ln n)^{3+\frac{2}{d+1}}n^{-\frac{1}{2}+\frac{1}{d+1}}$,
    which tends to zero for all $d\geq 2$, as $n\to\infty$.
    In a similar spirit, Theorems \ref{thm:Sphere}, \ref{thm:Hyperbolic}, \ref{thm:HilbertGeometries} and \ref{thm:dualVolume}
    can be upgraded to quantitative central limit theorems with the same bound on the Wasserstein distance.
\end{remark}

\section{Background material}\label{Background}

\subsection{Weighted floating bodies and geometric lemmas}

Let $K\in\cK(\RR^d)$, $\varphi\in\cW(K)$ and $\delta>0$.
The weighted floating body $K_\delta^\varphi$ of $K$ with respect to $\varphi$ and $\delta$ is defined as the set
\begin{equation*}
    K_\delta^\varphi := \bigcap \{ H^- : H\subset\RR^d\text{ a hyperplane, } \Phi(K\cap H^+)\leq\delta \},
\end{equation*}
where $H^\pm$ are the two closed half-spaces determined by a hyperplane $H\subset\RR^d$.
This concept was introduced in \cite{Werner2002} and further studied in \cite{BesauLudwigWerner}.
It generalizes the classical notion of convex floating bodies, which arises by taking $\varphi=\Vol(K)^{-1}$.
In this case we shall use the notation $K_\delta$ for the classical convex floating body of $K$ for parameter $\delta$.

The weighted floating body $K_\delta^\varphi$ is a convex body with non-empty interior for $\delta\in (0, \alpha(K,\varphi))$, where
\begin{equation*}
    \alpha(K,\varphi) = \max_{x\in K} f_K^\varphi(x),
\end{equation*}
and $f_K^\varphi:K\to (0,1)$ is the minimal cap measure, i.e.,
\begin{equation*}
    f_K^\varphi(x) = \min \{ \Phi(K\cap H^+) : x\in H^+\}.
\end{equation*}
In fact, we may equivalently define the weighted floating body $K_\delta^\varphi$ via the superlevel sets of $f_K^\varphi$, namely
\begin{equation*}
    K_\delta^\varphi = \{x\in K : f_K^\varphi(x) \geq \delta\}.
\end{equation*}
We immediately find that $K_{\delta_1}^\varphi \subset K_{\delta_2}^\varphi$ whenever $\delta_1\geq \delta_2$ and $K_\delta^\varphi \to K$ in the Hausdorff distance between convex bodies, as $\delta\to 0^+$.

Since $\varphi\in\cW(K)$ there is $\delta_0\in (0,\alpha(K,\varphi))$ such that $L\subset K_{\delta_0}^\varphi$,
where $L\subset \interior K$ is a convex body such that $\varphi$ is continuous and bounded on $K\setminus L$.
The constant $\delta_0$ is determined by
\begin{equation*}
    \delta_0 \leq \min_{u\in \SS^{d-1}} \Phi(K\cap H^+(u,h_L(u))),
\end{equation*}
i.e., for all $x\in L$ we have $f_K^\varphi(x) \geq \delta_0$ and therefore $L\subset K_{\delta_0}^\varphi$.
Thus if $\varphi\in\cW(K)$, then there exists $\delta_0\in (0,\alpha(K,\varphi))$ and $c,C>0$ such that for all $\delta \in (0,\delta_0)$, $\varphi$ is continuous and $c\leq \varphi(x) \leq C$ for all $x$ in $K\setminus K_\delta^\varphi$.

The next property of weighted floating bodies will turn out to be crucial for us. For a proof we refer to \cite[Lemma 5.2]{BesauLudwigWerner}.
It allows to compare a weighted floating body with suitable unweighted floating bodies.
\begin{lemma}\label{lem:ComparisonWeightedVSunweighted}
    Let $K\in\cK(\RR^d)$ and $\varphi\in\cW(K)$.
    Then there exist constants $c,\delta_0\in (0,1)$ independent from $\delta$ such that for all $\delta \in (0,\delta_0)$ one has that
    \begin{equation*}
        K_{\delta/c}
        \subset K_\delta^\varphi
        \subset K_{c\delta}\,.
    \end{equation*}
\end{lemma}

We also need frequently the behaviour of the volume of $K\setminus K_\delta$, as $\delta\to\infty$. The following fact can be found in \cite{BaranySurvey}, for example.
\begin{lemma}\label{lem:VolumeWetPart}
    Let $K\in\cK_+^2(\RR^d)$. Then there exist constants $c,\delta_0\in(0,1)$ such that for all $\delta\in (0,\delta_0)$,
    \begin{equation*}
        c \delta^{\frac{2}{d+1}}\leq \Vol(K\setminus K_\delta) \leq \tfrac{1}{c} \delta^{\frac{2}{d+1}}.
    \end{equation*}
\end{lemma}

We now rephrase a result taken from \cite[Lemma 4.2]{VuConcentration}, which shows that the random convex hull
$K_\varphi(n)$ generated by $n$ independent random points in a convex body $K\in\cK(\RR^d)$ with probability density
$\varphi\in\cW(K)$ contains the weighted floating body $K_\delta^\varphi$ with high probability if $\delta$ is essentially of order $\frac{\ln n}{n}$.

\begin{lemma}\label{lem:VanVu}
    Fix $K\in\cK(\RR^d)$, $\varphi\in\cW(K)$, $n\in\NN$ and $K_\varphi(n)$ be the random polytope as defined in \eqref{eq:DefRandomPolytopes}.
    For any $\beta\in(0,\infty)$, there exist constants $c,N\in(0,\infty)$ such that, for all $n\geq N$,
    \begin{equation*}
            \PP\Big(K_{c\frac{\ln n}{n}}^\varphi \not \subset K_\varphi(n)\Big) \leq n^{-\beta}.
    \end{equation*}
\end{lemma}

In the final part of this subsection we recall some geometric constructions that are of a more technical nature, but which will be needed in our further arguments. First, we define the visibility region
\begin{equation*}
    \Delta(z,\delta) := \{x\in K\setminus \interior K_\delta: [z,x] \cap K_\delta^\varphi = \emptyset\},
\end{equation*}
for $\delta\geq 0$ and $z\in K$, i.e., $\Delta(z,\delta)$ is the set of all points $x\in K$ that can be seen from $z$ without passing through $K_\delta$, see Figure \ref{fig:vis_region}.
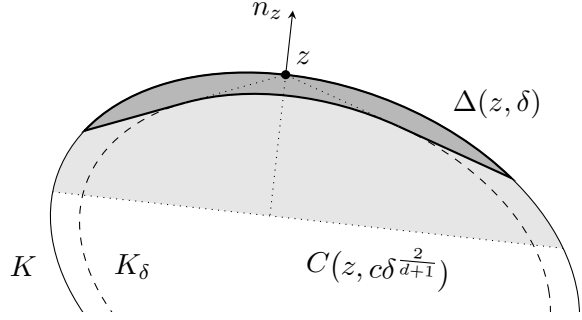
\begin{figure}[t]
    \centering
    \begin{tikzpicture}[scale=1.2]
        \clip (-3.5,-0.5) rectangle (3.5,3);
        \path[use as bounding box] (-3.5,-0.5) rectangle (3.5,3);
        \begin{scope}[rotate=-20, xscale=3, yscale=2, rotate=20] 

        \def\r{0.89}

        \draw (0,0) circle (1);
        \draw[dashed] (0,0) circle (\r);

        \coordinate (O) at (0,1);

        \coordinate (A1) at ({\r*sqrt(1-\r*\r)}, {\r*\r});
        \coordinate (B1) at ({2*\r*sqrt(1-\r*\r)},{2*\r*\r-1});

        \coordinate (A2) at ({-\r*sqrt(1-\r*\r)}, {\r*\r});
        \coordinate (B2) at ({-2*\r*sqrt(1-\r*\r)},{2*\r*\r-1});

        \fill[color=black, opacity=0.2] (A2)
            arc({180-atan(\r/sqrt(1-\r*\r))}:{atan(\r/sqrt(1-\r*\r))}:\r) -- (B1)
            arc({atan((2*\r*\r-1)/(2*\r*sqrt(1-\r*\r))}:{180-atan((2*\r*\r-1)/(2*\r*sqrt(1-\r*\r))}:1) -- cycle;
        \draw[thick] (A1) -- (B1) arc({atan((2*\r*\r-1)/(2*\r*sqrt(1-\r*\r))}:{180-atan((2*\r*\r-1)/(2*\r*sqrt(1-\r*\r))}:1) -- (A2);
        \draw[thick, white]  (A2) arc({180-atan(\r/sqrt(1-\r*\r))}:{atan(\r/sqrt(1-\r*\r))}:\r);
        \draw[thick] (A2) arc({180-atan(\r/sqrt(1-\r*\r))}:{atan(\r/sqrt(1-\r*\r))}:\r);
        \draw[dotted] (A1) -- (O) -- (A2);

        \begin{scope}[yshift=1cm]
            \fill[black] (0,0) [rotate=-20, xscale = 1/3, yscale = 1/2, rotate=20] circle(0.05) node[above right] {$z$};
        \end{scope}

        \coordinate (O1) at ([rotate=-40, xscale=1/3, yscale = 1/2, rotate=40] 0,0.8);
        \coordinate (O2) at ([rotate=0, xscale=1/3, yscale = 1/2, rotate=0] 4,0  );
        \draw[->,>=stealth] (O) --++ (O1) node[left] {$n_z$};

        \draw[dotted] (O) --++ ([rotate=180, scale=2.2] O1);
        \begin{scope}
            \clip circle (1);
            \fill[opacity=0.1] (O) ++ ([rotate=180, scale=2.2] O1) ++ (O2)
                --++ ([scale=-2]O2)  --++ ([scale=4.4] O1) --++ ([scale=2] O2) -- cycle;
            \draw[dotted] (O) ++ ([rotate=180, scale=2.2] O1) ++ (O2) --++ ([scale=-2]O2);
        \end{scope}

        \end{scope}

        \node[rotate=-5] at (0.7,0.) {$C\big(z,c\delta^{\frac{2}{d+1}}\big)$};
        \node at (-3.2,0.) {$K$};
        \node at (-2,0.) {$K_\delta$};
        \node at (2,1.8) {$\Delta(z,\delta)$};
    \end{tikzpicture}

    \caption{The visibility region $\Delta(z,\delta)$ and the cap $C\big(z,c\delta^{\frac{2}{d+1}}\big)$.}
    \label{fig:vis_region}
\end{figure}

Observe that for any $x\in \Delta(z,\delta)$ we have $x\not\in\interior K_\delta$ and therefore there is at least one hyperplane $H$ such that $x,z \in H^+$ and $\Vol(K\cap H^+)=\delta$, see e.g.\ \cite[Lemma 2]{SchuttWerner:1994}.
Thus $\Delta(z,\delta)$ is exactly the union of all caps $K\cap H^+$ that contain $z$ and cut off volume $\delta$ from $K$, i.e.,
\begin{equation*}
    \Delta(z,\delta) = \bigcup\left\{ K\cap H^+ : \text{$z\in H^+$ and $\Vol(K\cap H^+)=\delta$}\right\}.
\end{equation*}

The next series of geometric lemmas relies on the observation that a convex body $K\in\cK_+^2(\RR^d)$ locally looks like a ball from an equi-affine point of view, that is, there exists $r,R,t_0> 0$ such that for all $z\in \bd K$ we can find a volume preserving affine map $A_z$ that maps
$z$ to the origin, the normal direction $n_z$ is mapped to the coordinate direction $e_d$ and
$K$ is mapped to $K_z:=A_z(K)$ such that
\begin{equation*}
    C^{B(r)}(o,t) \subset C^{K_z}(o,t) \subset C^{B(R)}(o,t) \quad \text{for all $t\in[0,t_0]$},
\end{equation*}
where $B(s)= sB_2^d - se_d$ and $C^L(o,t) = \{x\in L : x\cdot e_d \geq - t\}$, see Figure \ref{fig:approx_ball}.
Furthermore, we can choose the affine map $A_z$ in such a way that there is no dilation in the normal direction, i.e., caps in direction $n_z$ of height $t$ will be mapped to caps in direction $e_d$ of height $t$. By approximating $K_z$ with the balls $B(r)$ and $B(R)$ we may then derive bounds uniformly for all $z\in \bd K$.

First, for a ball we may calculate that any cap of height $t$ has volume asymptotically of order $t^{\frac{d+1}{2}}$, as $t\to 0^+$.
For $K\in \cK_+^2(\RR^d)$ we therefore obtain the following uniform bound.

\begin{lemma}\label{lem:bound1}
    Let $K\in \cK_+^2(\RR^d)$ and $z\in \bd K$. Denote by $n_z$ the outer unit normal vector of $\bd K$ at $z$ and denote by $C(z,t)$ the cap of $K$ in direction $n_z$ of height $t$, that is,
        \begin{equation}\label{eq:DefCap}
            C(z,t) = \{x\in K : (x-z)\cdot n_z \geq -t\}.
        \end{equation}
    Then there exist $c_1,c_2,t_0\in (0,\infty)$ such that for all $t\in (0,t_0)$
    and all $z\in \bd K$ we have that
    \begin{equation*}
        c_1 t^{\frac{d+1}{2}}
        \leq \Vol\left(C(z,t)\right)
        \leq c_2 t^{\frac{d+1}{2}}.
    \end{equation*}

\end{lemma}

A more precise statement was obtained by Leichtweiss \cite[Hilfssatz~2]{Leichtweiss:1986}, who showed that for $K\in\cK_+^2(\RR^d)$ one actually has that
\begin{equation}\label{eqn:cap_limit}
    \lim_{t\to 0^+} \frac{\Vol\left(C(z,t)\right)}{t^{\frac{d+1}{2}}} = 2^{\frac{d+1}{2}} \frac{\Vol(B_2^{d-1})}{d+1} H_{d-1}(z)^{-\frac{1}{2}}
    \quad \text{for all $z\in \bd K$.}
\end{equation}

Next, we observe that the visibility region $\Delta(z,\delta)$ can be bounded by caps of height asymptotically of order $\delta^{2/(d+1)}$, as $\delta\to 0^+$.

\begin{figure}[t]
    \centering
    \begin{tikzpicture}[scale=2.8]
        \def\t0{0.68}
        \def\r{0.6}
        \def\R{2}

        \path[use as bounding box] (-2.5,-1.3) rectangle (2.5,0.5);

        \begin{scope}
        \clip (-2,-1) rectangle (2,0);

        \begin{scope}
            \clip (0,-\R) circle(\R);
            \fill[black!10] (-2,-\t0) rectangle (2,0);
            \draw[dotted] (-2,-\t0) rectangle (2,0);
        \end{scope}

        \begin{scope}
            \clip[domain=-2:2, variable=\x, smooth] plot ({1.44*\x}, {0.35*(\x*\x*\x-3*\x*\x)}) -- cycle;
            \fill[black!20] (-2,-\t0) rectangle (2,0);
            \draw (-2,-\t0) -- (2,-\t0);
        \end{scope}
        \draw[black!50] (0,-\r) circle(\r);
        \draw[black!50] (0,-\R) circle(\R);
        \draw[domain=-2:2, variable=\x, smooth, thick] plot ({1.44*\x}, {0.35*(\x*\x*\x-3*\x*\x)});
        \end{scope}

        \draw[->, >=stealth] (-2, 0) -- (2,0) node[above] {$\RR^{d-1}$};
        \draw[->, >=stealth] (0, -1) -- (0,0.3) node[left] {$e_d$};

        \fill[black] (0,0) circle(0.02);

        \node[above left] at (0,0) {$o$};

        \node[color=black!70] at (1.3,-0.3)   {$B(R)$};
        \node[color=black!70] at (0.7,-0.9)   {$B(r)$};
        \node at (-1,-0.9) {$K_z$};

        \node at (0.9, -0.58) {$C^{K_z}(o,t)$};

        \node[below left] at (0,-\t0) {$t$};
    \end{tikzpicture}
    \caption{Illustration of the observation that for any convex body $K\in \cK_+^2(\RR^d)$ we can find $r,R> 0$ such that for all $z\in \bd K$ there is a volume preserving affine transformation that maps $z$ to the origin $o$ and $K$ into a position $K_z$ between the balls $B(r)$ and $B(R)$ locally around $o$.}
    \label{fig:approx_ball}
\end{figure}
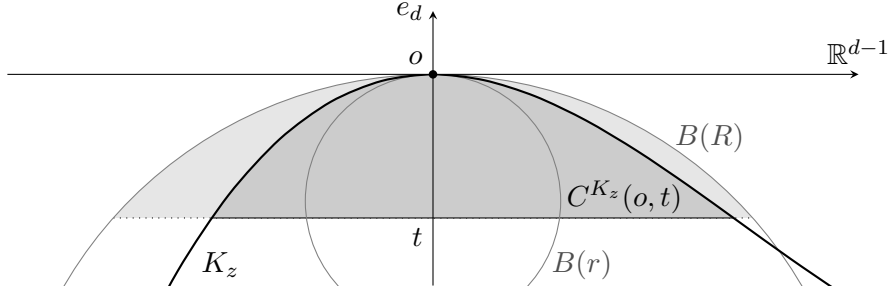

\begin{lemma}\label{lem:bound2}
    Let $K\in\cK_+^2(\RR^d)$. Then there exist $c_1,c_2,\delta_0\in (0,\infty)$ such that for all $\delta\in (0,\delta_0)$ and $z\in\bd K$ we have that
    \begin{equation*}
        C(z, c_1\delta^{\frac{2}{d+1}})
        \subset \Delta(z,\delta)
        \subset C(z,c_2\delta^{\frac{2}{d+1}}),
    \end{equation*}
    where $C(z,c_2\delta^{\frac{2}{d+1}})$ is defined as in \eqref{eq:DefCap}.
\end{lemma}
\begin{proof}
    Let $z\in \bd K$ be arbitrary.
    Notice that for $s\leq t$ we have $C(z,s)\subset C(z,t)$ and by continuity there is $t_z(\delta)$ such that $\Vol(C(z,t_z(\delta))) = \delta$.
    Then $C(z,t_z(\delta)) \subset \Delta(z,\delta)$.
    Recall that by \eqref{eqn:cap_limit}
    \begin{equation*}
        \lim_{\delta\to 0^+} \frac{t_z(\delta)}{\delta^{\frac{2}{d+1}}}
        = \frac{1}{2} \left(\frac{n+1}{\Vol(B_2^{d-1})}\right)^{\frac{2}{d+1}} H_{d-1}(z)^{\frac{1}{d+1}} \quad \text{for all $z\in \bd K$.}
    \end{equation*}
    By compactness of $\bd K$ and since $t_z(\delta) = t(z,\delta)$ is monotone in $\delta$ and continuous in both arguments,
    we find $c_1,\delta_1\in (0,\infty)$ such that for all $z\in\bd K$ and $\delta\in (0,\delta_1)$, we have
    \begin{equation*}
        t_z(\delta) \geq c_1 \delta^{\frac{2}{d+1}},
    \end{equation*}
    which yields $\Delta(z,\delta) \supset C(z, c_1 \delta^{\frac{2}{d+1}})$ for all $z\in\bd K$ and $\delta\in (0,\delta_1)$.

    \medskip
    For the upper bound we repeat the construction as illustrated in Figure \ref{fig:approx_ball}. First set
    \begin{equation*}
        r:= \tfrac{1}{2} \min_{z\in\bd K} H_{d-1}(z)^{\frac{2}{d-1}}
        \quad \text{and}\quad
        R:= 2 \max_{z\in \bd K} H_{d-1}(z)^{\frac{2}{d-1}}.
    \end{equation*}
    Now let $z\in \bd K$ be arbitrary. We consider the volume preserving affine transformation $\alpha_z^K$ defined by
    \begin{equation*}
        \alpha_z^K(e_i) = z+ \frac{\kappa_i(z)}{\sqrt{\rho_z}} v_i, \quad \text{for $i=1,\dotsc,d-1$},
    \end{equation*}
    and $\alpha_z^K(e_d) = z + n_z$, where $\kappa_i(z)$ are the principal curvatures of $\bd K$ at $z$ and $\{v_i\}_{i=1}^{n-1}$ is the corresponding orthonormal basis of principal directions and where we set
    \begin{equation*}
        \rho_z := H_{d-1}(z)^{\frac{2}{d-1}} = \prod_{i=1}^{d-1} \kappa_i(z)^{\frac{2}{d-1}}.
    \end{equation*}
    We further put $B(s) := B_2^d(-s e_d, s)$, i.e, $B(s)$ is the ball of radius $s$ such that $o$ is a boundary point with outer unit normal $e_d$.
    The ball $B(\rho_z)$ is transformed by $\alpha_z^K$ into the standard approximating ellipsoid $\cE$ of $K$ at $z$, i.e.,
    we have that
    \begin{equation*}
        \kappa_i^\cE(z) = \kappa_i^K(z),\quad \text{for $i=1,\dotsc,d-1$,}
    \end{equation*}
    see \cite[Sec.\ 1.6]{SchuttWerner:2003}.
    Notice that by our choice of $r$ and $R$ we have $B(r) \subset B(\rho_z)\subset B(R)$ for all $z\in\bd K$.
    Now set $K_z := (\alpha_z^K)^{-1}(K)$.
    Again, by the choice of $r$ and $R$ there exists $t_0\in (0,\infty)$ independent of $z\in \bd K$ such that for all $t\in [0,t_0]$ we have that
    \begin{equation*}
        C^{B(r)}(o,t) \subset C^{K_z}(o,t) \subset C^{B(R)}(o,t).
    \end{equation*}
    For the ball $B(r)$ we may verify that there are $c_2,\delta_2\in (0,\infty)$ such that for all $\delta\in (0,\delta_2)$ we have that if $B(r)\cap H^+$
    is a cap of volume $\delta$ such that $o\in H^+$, then
    \begin{align*}
        B(r)\cap H^+&\subset C^{B(r)}(o,c_2\delta^{\frac{2}{d+1}}) = \{x\in B(r) : x\cdot e_d \geq - c_2\delta^{\frac{2}{d+1}}\}.
     \end{align*}
	Since $B(R)\cap H^+ \subset \frac{R}{r} (B(r)\cap H^+)$, this also yields
	\begin{align*}
        B(R)\cap H^+ &\subset C^{B(R)}(o,c_2(R/r)\delta^{\frac{2}{d+1}}).
    \end{align*}
    Thus for all $z\in \bd K$ and $\delta \in (0,\delta_2)$ we derive that
    \begin{align*}
        \Delta(z,\delta)
        &=\bigcup\{ K\cap H^+ : \text{$z\in H^+$ and $\Vol(K\cap H^+)=\delta$}\}\\
        &\subset \alpha_z^K \left[K_z\cap \bigcup\{B(R) \cap H^+ : \text{$o\in H^+$ and $\Vol(B(r)\cap H^+)=\delta$}\}\right]\\
        &\subset \alpha_z^K \left[K_z \cap C^{B(R)}\left(o,c_2(R/r)\delta^{\frac{2}{d+1}}\right)\right]\\
        &= C(z,c_2(R/r)\delta^{\frac{2}{d+1}}).
    \end{align*}
    The lemma is now complete by setting $\delta_0:=\min\{\delta_1,\delta_2\}$.
\end{proof}

As a consequence, Lemma \ref{lem:bound1} and Lemma \ref{lem:bound2} yield that there exist $c_1,c_2,\delta_0\in (0,\infty)$ such that for all $\delta\in (0,\delta_0)$ and $z\in\bd K$,
\begin{equation*}
    c_1 \delta \leq \Vol(\Delta(z,\delta)) \leq c_2 \delta.
\end{equation*}
Moreover, the volume of the union of $\delta$ volume caps that intersect a fixed $\delta$ volume cap is bounded above by a constant multiple of $\delta$ as $\delta \to 0^+$, see e.g.\ \cite[Lemma 6.3]{VuConcentration}. We therefore have the following corollary to Lemma \ref{lem:bound1} and Lemma \ref{lem:bound2}.

\begin{corollary}\label{cor:est_vis_region}
    Let $K\in\cK_+^2(\RR^d)$. Then there exist $c_1,c_2,\delta_0\in (0,\infty)$ such that for all $\delta \in (0,\delta_0)$ and for all $z\in \bd K$,
    \begin{equation}\label{eqn:upper_bound_union}
        c_1\delta \leq \Vol(\Delta(z,\delta)) \leq \Vol\left(\bigcup\left\{ \Delta(x,\delta) : x\in \Delta(z,\delta) \right\} \right) \leq c_2\delta.
    \end{equation}
\end{corollary}

\subsection{Stein's method for normal approximation}

Stein's method is a well known and very flexible device for proving probabilistic limit theorems. In this paper we shall use a version of Stein's method for normal approximation, which was originally introduced in \cite{Chatterjee}. We shall use a variation of the main result from this work taken from \cite{LachPecc}, which has turned out to be particularly useful for applications in stochastic geometry.

To introduce the set-up we denote by $\XX$ a Polish space and consider a fixed Borel probability measure $\mu$ on $\XX$.
For $n\in\NN$, let $f:\bigcup_{k=1}^n\XX^k\to\RR$ be a symmetric measurable function.
That is, $f$ is a symmetric function acting on configurations of at most $n$ points of $\XX$. If $x=(x_1,\ldots,x_n)\in\XX^n$ and $i\in\{1,\ldots,n\}$, then
\begin{equation*}
    x_{\neg i}:= (x_1,\dotsc,x_{i-1}, x_{i+1},\dotsc,x_n)\in\XX^{n-1},
\end{equation*}
will denote the $(n-1)$-dimensional vector, which is obtained from $x$ be removing the $i$th coordinate $x_i$.
Similarly, for $i,j\in\{1,\ldots,n\}$ with $i<j$ we set $x_{\neg i,j}\in\XX^{n-2}$ to be the $(n-2)$-dimensional vector arising from $x$ by removing
$x_i$ and $x_j$. The first- and second-order difference operators of $f$ are defined by
\begin{align*}
    D_i f(x) &:= f(x)-f(x_{\neg i}),
\intertext{and}
    D_{i,j} f(x) &:= D_i(D_j f(x)) = f(x)-f(x_{\neg i})-f(x_{\neg j})+f(x_{\neg i,j}).
\end{align*}

Now let $X=(X_1,\ldots,X_n)$ be a random vector distributed with respect to $\mu$ with coordinates $X_1,\ldots,X_n\in\XX$, and let $X'$ and $X''$
be independent random copies of $X$ with coordinates $X_i'$ and $X_i''$, $i\in\{1,\ldots,n\}$, respectively.
By a recombination of $\{X,X'\!\!,X''\}$ we understand a random vector $Z=(Z_1,\ldots,Z_n)$ such that $Z_i\in\{X_i,X_i',X''_i\}$ for all $i\in\{1,\ldots,n\}$.
Using the notion of recombinations we can now introduce the following four quantities, which will turn out to play an important role:
\begin{align*}
    \gamma_1 & := \sup_{(Y,Y'\!\!,Z,Z')}\EE\big[\mathbf{1}\{D_{1,2}f(Y)\neq 0\}\mathbf{1}\{D_{1,3}f(Y')\neq 0\}\,(D_2f(Z))^2(D_3f(Z'))^2\big]\,,\\
    \gamma_2 & := \sup_{(Y,Z,Z')}\EE\big[\mathbf{1}\{D_{1,2}f(Y)\neq 0\}\,(D_1f(Z))^2(D_2f(Z'))^2\big]\,,\\
    \gamma_3 &:= \EE|D_1f(X)|^3\,,\\
    \gamma_4 &:= \EE|D_1f(X)|^4\,,
\end{align*}
where in the definition of $\gamma_1$ the supremum is taken over all $4$-tuples of random vectors $Y, Y'\!\!, Z$, and $Z'$,
which are recombinations of $\{X,X'\!\!,X''\}$, while in the definition of $\gamma_2$ the supremum is taken over all $3$-tuples of random vectors $Y$, $Z$ and $Z'$, which are recombinations of $\{X,X'\!\!,X''\}$.

\smallskip
Next, we recall the definition of the Wasserstein distance $d_{\Wass}(\cL(W),\cL(V))$
between the laws $\cL(W)$ and $\cL(V)$ of two real-valued random variables $W$ and $V$ on $\RR$.
Denoting by $\Lip_1$ the space of Lipschitz functions $f:\RR\to\RR$ with Lipschitz constant less than or equal to $1$ we have that
\begin{equation}\label{eq:DefWasserstein}
    d_{\Wass}(\cL(W),\cL(V)) := \sup_{g\in\Lip_1}\big|\EE g(W)-\EE g(V)\big|\,.
\end{equation}
Note that if $(W_n)_{n\in\NN}$ is a sequence of random variables satisfying $d_{\Wass}(\cL(W_n),\cL(V))\to 0$, as $n\to\infty$, then we have the convergence in distribution $W_n\overset{d}{\longrightarrow}V$, see \cite[Chapter 8.3]{Bogachev}.

\smallskip
We are now prepared to rephrase the following result from \cite{Chatterjee, LachPecc}.
\begin{lemma}\label{lem:CLTLachiezeReyPeccati}
    Fix $n\in\mathbb{N}$. Let $X_1,\dotsc,X_n$ be independent random vectors in a Polish space $\XX$ with respect to a Borel probability measure $\mu$ and let
    $f:\bigcup_{k=1}^n\XX^k\to\RR$ be a symmetric Borel measurable function.
    Define $W(n):=f(X_1,\ldots,X_n)$ and assume that $\EE\, W(n)=0$ and $\EE\, W(n)^2\in(0,\infty)$.
    Then there exists an absolute constant $c>0$ such that
    \begin{equation}\label{eq:BoundWassersteinGeneral}
    \begin{split}
        d_{\Wass}\left(\cL\left(\frac{W(n)}{\sqrt{\Var W(n)}}\right),\cL(Z)\right)
        &\leq \frac{c\sqrt{n}}{\Var W(n)}\left(\sqrt{n^2\gamma_1}+\sqrt{n\gamma_2}\right.\\
        &\hspace{3cm}\left.+ \sqrt{\frac{n}{\Var W(n)}}\, \gamma_3 +\sqrt{\gamma_4}\right),
    \end{split}
    \end{equation}
    where $Z$ is a standard Gaussian random variable.
    In particular, if the right hand side tends to zero, as $n\to\infty$, then $W(n) / \sqrt{\Var W(n)}\overset{d}{\longrightarrow}Z$.
\end{lemma}

\section{Proof of Theorem \ref{thm:WeightedVolumeEuclidean}}\label{sec:ProofWeighted}

\subsection{Further notation}

Before entering the details of the proof of Theorem \ref{thm:WeightedVolumeEuclidean} let us introduce some further notation, which is frequently applied below. We shall indicate by $H(u,t)$ the hyperplane in $\RR^d$ with unit normal direction $u\in\SS^{d-1}$ and signed distance $t\in \RR$ from the origin, i.e., $H(u,t) = \{x\in \RR^d: x\cdot u = t\}$. Moreover, $e_1,\ldots,e_d$ stands for the standard orthonormal basis in $\RR^d$.

If $Y,X_1,\ldots,X_m$ are random variables and $F=F(Y,X_1,\ldots,X_m)$ is a measurable function of these random variables,
then we write $\Var_Y F$ for the variance taken with respect to $Y$ only, that is,
for the conditional variance $\Var(F|X_1,\ldots,X_m)$ of $F$ given $X_1,\ldots,X_m$.

For two sequences $(a_n)_{n\in\NN}$ and $(b_n)_{n\in\NN}$ we write $a_n\sim b_n$ to indicate that $a_n/b_n\to c$, as $n\to\infty$, where $c\in(0,\infty)$ is some constant independent of $n$. Furthermore, we write $a_n\gtrsim b_n$ if there exists a constant $c>0$, which is independent of $n$, such that $a_n\geq cb_n$.

\subsection{A lower variance bound}

The first step in the proof of Theorem \ref{thm:WeightedVolumeEuclidean} is a lower bound on the variance of the weighted volume $\Psi(K_\varphi(n))$ of the weighted random polytopes $K_\varphi(n)$. For $\Phi$ and $\Psi$ being the uniform distribution on a convex body $K\in\cK_+^2(\RR^d)$ a lower variance bound was proved in \cite{ReitznerCLT2005}. Essentially following the ideas in \cite{ReitznerCLT2005} we shall extend this result to the weighted case. It turns out that the order of the lower bound for the variance is independent of the choice of the weight functions, they just affect the constant.

\begin{theorem}\label{thm:LowerVarianceBound}
    Fix $K\in\cK_+^2(\RR^d)$, $\varphi\in\cW(K)$, $n\in\NN$ and $K_\varphi(n)$ be the weighed random polytope defined as in \eqref{eq:DefRandomPolytopes}.
    Then there exist constants $c,N\in(0,\infty)$ such that for all $n\geq N$ we have
    \begin{equation*}
        \Var\Psi(K_\varphi(n))\geq c\,n^{-\frac{d+3}{d+1}}.
    \end{equation*}
\end{theorem}

In the following proof, $c,c',c_0,c_1$ etc.\ will denote positive and finite constants, which are independent from the parameter $n$.

\begin{proof}[Proof of Theorem \ref{thm:LowerVarianceBound}]
	\emph{Step 1 -- standard paraboloid.}
	Let $E$ be the standard paraboloid in $\RR^d$, that is,
	\begin{equation*}
		E = \{(x_1,\dotsc,x_d) \in \RR^d : x_1^2 + \dotsc + x_{d-1}^2 \leq x_d\}.
	\end{equation*}
	We chose a simplex $S_0$ in the cap
	\begin{equation*}
		C^{\tinyonehalf E}(o,1) = \left\{x\in \tfrac{1}{2}E: x_d\leq 1\right\}
	\end{equation*}
	in the following way.
	The base is a regular simplex $\Delta_{d-1}$ with vertices on the $(d-2)$-sphere $\bd(\frac{1}{2}E \cap H(e_d, s))$ and the apex is at the origin $o$, see Figure \ref{fig:proof_sketch}.
	Then the circumradius $R$ of $\Delta_{d-1}$ is equal to the radius of the $(d-1)$-dimensional ball $\frac{1}{2}E\cap H(e_d, s)$ and therefore $R=\sqrt{s/2}$. Thus the inradius of $\Delta_{d-1}$ is $r=(1/d)\sqrt{s/2}$. We choose $s$ such that the cone $\pos S_0$ generated by $S_0$ contains $2E\cap H(e_d,1)$. Because the inradius of that $(d-1)$-simplex, that is the intersection of $\pos S_0$ with $H(e_d,1)$, is $r/s = 1/(d\sqrt{2s})$ and $2E\cap H(e_d,1)$ is a $(d-1)$-dimensional ball of radius $1/\sqrt{2}$ we may choose $s=1/(32d^2)$.

	\begin{figure}[t]
        \centering
        \begin{tikzpicture}[scale=3]
            \def\offSet{0.1}
            \path[use as bounding box] (-2,-0.1) rectangle (2,1.6);
            \draw[domain={-1-\offSet}:{1+\offSet}, variable=\x, smooth, thick] plot (\x, {\x*\x})
                node[above] {$E$}; 
            \draw[domain={-1/sqrt(2)-\offSet}:{1/sqrt(2)+\offSet}, variable=\x, smooth, opacity=.5] plot (\x,{2*\x*\x})
                node[above] {$\frac{1}{2}E$}; 
            \draw[domain={-sqrt(2)-\offSet}:{sqrt(2)+\offSet}, variable=\x, smooth, opacity=.5] plot (\x,{0.5*\x*\x})
                node[above] {$2E$}; 

            \def\s{0.2}
            \fill[black, opacity=0.5] (0,0) -- ({-sqrt(0.5*\s)},\s) -- ({sqrt(0.5*\s)},\s) -- cycle;
            \draw[thick] (0,0) -- ({sqrt(0.5*\s)}, \s) -- ({-sqrt(0.5*\s)}, \s) -- cycle;
            \node[above] at (0,\s) {$S_0$};

            \draw[dashed, shorten >= -1cm] ({sqrt(0.5*\s)}, \s) -- ({sqrt(0.5/\s)},1)
                node[right] {$\pos S_0$};
            \draw[dashed, shorten >= -1cm] ({-sqrt(0.5*\s)}, \s) -- ({-sqrt(0.5/\s)},1);

            \draw[thick] ({-sqrt(2)},1) -- ({sqrt(2)},1) node[midway, above] {$2E\cap H(e_d,1)$};
        \end{tikzpicture}
        \caption{Sketch for the construction of the simplex $S_0$ in Step 1 of the proof of Theorem \ref{thm:LowerVarianceBound}.}
        \label{fig:proof_sketch}
	\end{figure}
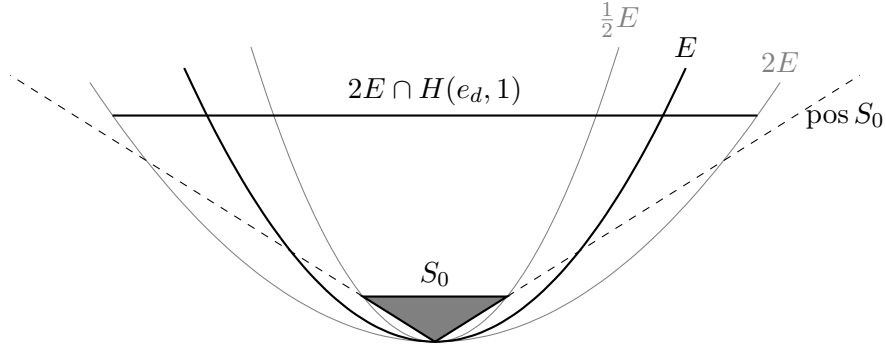

	By continuity there exists $\delta_0>0$ and closed sets $C^1,\dotsc,C^d \subset C^{\tinyonehalf E}(o,1)$ (e.g., suitable caps whose centers are the vertices of the base of $S_0$) such that
	\begin{equation*}
		\min\{\Vol(C^{\tinyonehalf E}(o,\delta_0)), \Vol(C^1), \dotsc, \Vol(C^d)\} =: c_0 >0,
	\end{equation*}
	and such that for all $Y\in C^{\tinyonehalf E}(o,\delta_0)$, $x_i\in C^i$, $i=1,\dotsc,d$, we have that the simplex $[Y,x_1,\dotsc,x_d]$ is `close' to $S_0$, in particular that
	\begin{equation*}
		\pos\, [Y,x_1,\dotsc,x_d] \supset 2E\cap H(e_d,1),
	\end{equation*}
	see also Figure \ref{fig:proof_sketch2}.
	Note that if $\bar{\varphi}:C^{{\scriptscriptstyle \frac{1}{2}} E}(o,1)\to (0,\infty)$ is an integrable function, then
	\begin{equation}\label{eqn:bound_phi_E}
		\min\{ \bar{\Phi}(C^{\tinyonehalf E}(o,\delta_0)), \bar{\Phi}(C^1), \dotsc, \bar{\Phi}(C^d)\} \geq c_0 \, \inf_{C^{\tinyonehalf E}(o,1)}\bar{\varphi}.
	\end{equation}

	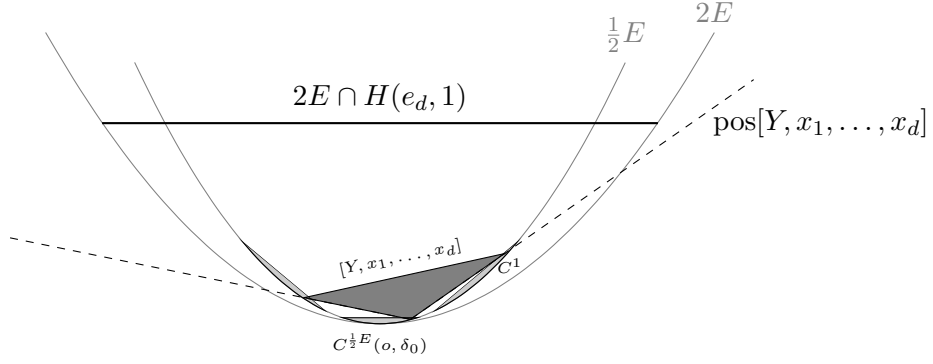
\begin{figure}[t]
        \centering
        \begin{tikzpicture}[xscale=4, yscale=2.66]
            \def\offSet{0.1}
            \path[use as bounding box] (-1.8,-0.1) rectangle (1.8,1.6);
            \draw[domain={-1/sqrt(2)-\offSet}:{1/sqrt(2)+\offSet}, variable=\x, smooth, opacity=.5] plot (\x,{2*\x*\x})
                node[above] {$\frac{1}{2}E$}; 
            \def\s{0.2} 
            \def\h{0.015} 
            \def\t{0.03} 
            \def\r{0.3} 

            \begin{scope}
                \clip ({-\r},{-\t}) rectangle (\r,\t);
                \fill[domain={-1/sqrt(2)-\offSet}:{1/sqrt(2)+\offSet}, variable=\x, smooth, opacity=.2] plot (\x,{2*\x*\x}) --cycle;
                \draw[domain={-1/sqrt(2)-\offSet}:{1/sqrt(2)+\offSet}, variable=\x, smooth] plot (\x,{2*\x*\x}) --cycle;
                \begin{scope}
                    \clip[domain={-1/sqrt(2)-\offSet}:{1/sqrt(2)+\offSet}, variable=\x, smooth] plot (\x,{2*\x*\x}) --cycle;
                    \draw ({-\r},\t) -- (\r,\t);
                \end{scope}
            \end{scope}

            \begin{scope}
                \clip ({sqrt(0.5*\s)},\s) ++ ({-\h*2*sqrt(2*\s)},\h) ++ (\r,{\r*2*sqrt(2*\s)})
                --++ ({-2*\r},{-\r*4*sqrt(2*\s)})
                --++ ({\h*4*sqrt(2*\s)},{-2*\h})
                --++ ({2*\r},{\r*4*sqrt(2*\s)}) -- cycle;
                \fill[domain={-1/sqrt(2)-\offSet}:{1/sqrt(2)+\offSet}, variable=\x, smooth, opacity=.2] plot (\x,{2*\x*\x}) --cycle;
                \draw[domain={-1/sqrt(2)-\offSet}:{1/sqrt(2)+\offSet}, variable=\x, smooth] plot (\x,{2*\x*\x}) --cycle;
                \begin{scope}
                    \clip[domain={-1/sqrt(2)-\offSet}:{1/sqrt(2)+\offSet}, variable=\x, smooth] plot (\x,{2*\x*\x}) --cycle;
                    \draw ({sqrt(0.5*\s)},\s) ++ ({-\h*2*sqrt(2*\s)},\h) ++ (\r,{\r*2*sqrt(2*\s)}) --++ ({-2*\r},{-\r*4*sqrt(2*\s)});
                \end{scope}
            \end{scope}

            \begin{scope}
                \clip ({-sqrt(0.5*\s)},\s) ++ ({\h*2*sqrt(2*\s)},\h) ++ (\r,{-\r*2*sqrt(2*\s)})
                --++ ({-2*\r},{\r*4*sqrt(2*\s)})
                --++ ({-\h*4*sqrt(2*\s)},{-2*\h})
                --++ ({2*\r},{-\r*4*sqrt(2*\s)}) -- cycle;
                \fill[domain={-1/sqrt(2)-\offSet}:{1/sqrt(2)+\offSet}, variable=\x, smooth, opacity=.2] plot (\x,{2*\x*\x}) --cycle;
                \draw[domain={-1/sqrt(2)-\offSet}:{1/sqrt(2)+\offSet}, variable=\x, smooth] plot (\x,{2*\x*\x}) --cycle;
                \begin{scope}
                    \clip[domain={-1/sqrt(2)-\offSet}:{1/sqrt(2)+\offSet}, variable=\x, smooth] plot (\x,{2*\x*\x}) --cycle;
                    \draw ({-sqrt(0.5*\s)},\s) ++ ({\h*2*sqrt(2*\s)},\h) ++ (\r,{-\r*2*sqrt(2*\s)}) --++ ({-2*\r},{\r*4*sqrt(2*\s)});
                \end{scope}
            \end{scope}

            \coordinate (A) at (0.1,0.022);
            \coordinate (B) at (0.41, 0.35);
            \coordinate (C) at (-0.25,0.13);

            \fill[opacity=0.5] (A) -- (B) -- (C) --cycle;
            \draw (A) -- (B) -- (C) node[midway,above, rotate=12]{\tiny $[Y,x_1,\dotsc,x_d]$} --cycle;
            \draw[dashed, shorten >= -4cm] (A) -- (B);
            \draw[dashed, shorten >= -4cm] (A) -- (C);

            \node[below] at (0,0) {\tiny $C^{\tinyonehalf E}(o,\delta_0)$};
            \node at (1.45,1) {$\pos [Y,x_1,\dotsc,x_d]$};

            \node[above right, xshift = 0.14cm] at ({sqrt(0.5*\s)},\s) {\tiny $C^1$};

            \def\offSetH{0.5}
            \draw[domain={-1-\offSet}:{1+\offSet}, variable=\x, smooth, opacity=0.5] plot (\x,{1.2*\x*\x})
                node[above] {$2E$}; 
            \draw[thick] ({-sqrt(2)+\offSetH},1) -- ({sqrt(2)-\offSetH},1) node[midway, above] {$2E\cap H(e_d,1)$};
        \end{tikzpicture}
        \caption{Sketch for the construction of the caps in Step 1 of the proof of Theorem \ref{thm:LowerVarianceBound}.}
        \label{fig:proof_sketch2}
	\end{figure}

	Next, we need the following lemma.
	\begin{lemma}\label{lem:AuxLowerBound}
	Let $\bar{\psi}:C^{\tinyonehalf E}(o,1)\to (0,\infty)$ be a continuous function that is bounded below and let $\vartheta$ be a continuous probability density on $C^{\tinyonehalf E}(o,\delta_0)$.  Then there is $c>0$ such that for all $x_i\in C^i$, $i=1,\dotsc,d$, we have
	\begin{equation}\label{eqn:bound_V_E}
		\Var_Y \bar\Psi([Y,x_1,\dotsc,x_d]) \geq c \, \inf_{C^{\tinyonehalf E}(o,\delta_0)} \vartheta,
	\end{equation}
	where $Y$ is distributed with respect to $\vartheta$.
	\end{lemma}
	\begin{proof}[Proof of Lemma \ref{lem:AuxLowerBound}]
	Note that the functional $(y,x_1,\dotsc,x_d) \mapsto \bar\Psi([y,x_1,\dotsc,x_d])$ is continuous, strictly positive and not constant on $C^{\tinyonehalf E}(o,\delta_0)\times C^1\times\ldots\times C^d$. Since $\bar{\Psi}$ is continuous and $C^1,\ldots,C^d$ are compact there exists $z\in C^{\tinyonehalf E}(o,\delta_0)$ and $\delta>0$ such that
	\begin{align*}
		c':= \min_{z'\in B_2^d(z,\delta)\,\cap\, C^{\tinyonehalf E}(o,\delta_0)} \min_{x_i\in C^i} \left(\bar\Psi([z',x_1,\dotsc,x_d]) - \EE_Y \bar\Psi([Y,x_1,\dotsc,x_d]) \right)^2 > 0,
	\end{align*}
	where $B_2^d(z,\delta)$ is the Euclidean ball around $z$ with radius $\delta$.
	Thus,
	\begin{align*}
		\Var_Y \bar\Psi([Y,x_1,\dotsc,x_d])
		&= \int_{C^{\tinyonehalf E}(o,\delta_0)} (\bar\Psi([Y,x_1,\dotsc,x_d])- \EE_Y\bar\Psi([Y,x_1,\dotsc,x_d]))^2 \, \dint Y \\
		&\geq \int_{B_2^d(z,\delta)\,\cap\, C^{\tinyonehalf E}(o,\delta_0)} (\bar\Psi([Y,x_1,\dotsc,x_d])- \EE_Y\bar\Psi([Y,x_1,\dotsc,x_d]))^2\, \dint Y\\
		&\geq c'\, \Vol(B_2^d(z,\delta)\cap C^{\tinyonehalf E}(o,\delta_0)) \,\inf_{C^{\tinyonehalf E}(o,\delta_0)} \vartheta.
	\end{align*}
    This proves the lemma.
	\end{proof}

	We can now continue the proof of Theorem \ref{thm:LowerVarianceBound}.

	\bigskip

	\emph{Step 2 -- elliptic paraboloid of height $h$.}
	Let $\kappa_1,\dotsc,\kappa_{d-1}>0$, and set $Q$ to be the elliptic paraboloid
	\begin{equation*}
		Q := \left\{(x_1,\dotsc,x_d)\in\RR^d : x_d \geq \tfrac{1}{2} \sum_{i=1}^{d-1} \kappa_j x_j^2\right\},
	\end{equation*}
	and put $\kappa=\prod_{i=1}^{d-1} \kappa_j$.
	The linear map $\alpha$ defined by
	\begin{equation*}
		\alpha(e_i) = \sqrt{\frac{2h}{\kappa_i}}\  e_i\quad (i\in\{1,\ldots,d-1\})\qquad\text{and}\qquad \alpha(e_d) = he_d,
	\end{equation*}
	maps $C^{\tinyonehalf E}(o,\delta_0)$ to $C^{\tinyonehalf Q}(o,h\delta_0)$ and satisfies
	\begin{equation*}
		\det \alpha = 2^{\frac{d-1}{2}} \kappa ^{-\frac{1}{2}} h^{\frac{d+1}{2}} \sim h^{\frac{d+1}{2}},\qquad \text{as $h\to 0^+$}.
	\end{equation*}
	Notice that if $\varphi$ and $\psi$ are a integrable functions on $C^{\tinyonehalf Q}(o,h\delta_0)$
	that are bounded from below by some constant then so are $\bar{\varphi}:=\varphi\circ\alpha$ and
	$\bar{\psi}:=\psi\circ \alpha$, and for any Borel set $B\subset C^{\tinyonehalf E}(o,\delta_0)$ we have that
	\begin{equation}\label{eqn:alpha_transf}
		\Phi(\alpha(B)) = (\det \alpha)\int_{B} (\varphi\circ \alpha)(x)\,\dint x = (\det \alpha) \bar{\Phi}(B)
		\qquad\text{and}\qquad
		\Psi(\alpha(B)) = (\det \alpha) \bar{\Psi}(B).
	\end{equation}
	Applying $\alpha$ to $C^1,\dotsc,C^d$, we obtain closed sets $D^i=\alpha(C_i)$ for $i\in\{1,\ldots,d\}$ satisfying
	\begin{align*}
		\min\{\Phi(C^{\tinyonehalf Q}(o,h\delta_0)),\Phi(D^1), \dotsc, \Phi(D^d)\}
		&= (\det\alpha)\min\{\bar\Phi(C^{\tinyonehalf E}(o,\delta_0)),\bar\Phi(C^1), \dotsc, \bar\Phi(C^d)\} \\
		&\geq c_0 \,(\det \alpha) \,\inf_{C^{\tinyonehalf E}(o,1)} \bar{\varphi}\\
		&= c_0 \, (\det \alpha) \,\inf_{C^{\tinyonehalf Q}(o,h)} \varphi.
	\end{align*}
	Now observe that, by \eqref{eqn:bound_phi_E}, the last expression behaves like $h^{\frac{d+1}{2}}$, as $h\to 0^+$.
	Moreover, for all $y\in C^{\tinyonehalf Q}(o,h\delta_0)$ and $x_i\in D^i$, we have that
	\begin{equation*}
		\operatorname{pos}\, [y,x_1,\dotsc,x_d] \supset 2Q\cap H(e_d,h).
	\end{equation*}
 	Let $\varphi$ be bounded from below and consider the probability density $\vartheta$ that is the restriction of $\varphi$ to $C^{\tinyonehalf Q}(o,h\delta_0)$, i.e.,
 	\begin{equation*}
 		\vartheta(x) = \frac{\mathbf{1}_{C^{\tinyonehalf Q}(o,h\delta_0)}(x)}{\Phi(C^{\tinyonehalf Q}(o,h\delta_0))} \varphi(x).
 	\end{equation*}
	Let $Y$ be a random point distributed according to a probability density $\vartheta$ on $C^{\tinyonehalf Q}(o,h\delta_0)$ and let $x_i\in D^i$ for each $i\in\{1,\ldots,d\}$.
	Then $\bar{Y}:=\alpha^{-1}(Y)$ is a random point in $C^{\tinyonehalf E}(o,\delta_0)$ that has the density $\tilde{\vartheta} := (\det \alpha)\, \vartheta\circ\alpha$.
	We set $\bar{x}_i:=\alpha^{-1}(x_i)$ for $i\in\{1,\ldots,d\}$. Then,
	\begin{equation*}
		\Var_Y \Psi([Y,x_1,\dotsc,x_d])
		= \Var_{\bar{Y}} \Psi(\alpha[\bar{Y},\bar{x}_1,\dotsc,\bar{x}_d])
		=(\det \alpha)^2 \Var_{\bar{Y}} \bar{\Psi}([\bar{Y},\bar{x}_1,\dotsc,\bar{x}_d]),
	\end{equation*}
	where we used \eqref{eqn:alpha_transf} in the last step.
	By \eqref{eqn:bound_V_E}, this yields
	\begin{align*}
		\Var_Y \Psi([Y,x_1,\dotsc,x_d]) &\geq c \Big(\inf_{C^{\tinyonehalf E}(o,\delta_0)} \!\!\tilde{\vartheta}\Big) (\det \alpha)^2
		= c  \Big(\inf_{C^{\tinyonehalf Q}(o,h\delta_0)}\!\! \vartheta\Big) (\det \alpha)^3\\
		&= \frac{c}{\Phi(C^{\tinyonehalf Q}(o,h\delta_0))}  \Big(\inf_{C^{\tinyonehalf Q}(o,h\delta_0)}\!\! \varphi\Big) (\det \alpha)^3\\
		&\geq \frac{c}{\Vol(C^{\tinyonehalf E}(o,\delta_0))}
            \Bigg(\inf\limits_{C^{\tinyonehalf Q}(o,h\delta_0)}\!\! \varphi \Bigg/\!\!\! \sup\limits_{C^{\tinyonehalf Q}(o,h\delta_0)} \!\!\varphi \Bigg) (\det\alpha)^2,
	\end{align*}
	and the last term behaves like $h^{d+1}$, as $h\to 0^+$.

	\bigskip

	\emph{Step 3 -- economically cover $K$ with caps and approximate by elliptic paraboloids.}
	By our assumptions on $K$ the Gauss--Kronecker curvature exists in all boundary points and is bounded from below by a positive constant.
	Now set
	\begin{equation*}
		m := \lfloor n^{\frac{d-1}{d+1}}\rfloor,
	\end{equation*}
	i.e., $m$ is the largest integer less than or equal to $n^{\frac{d-1}{d+1}}$. Also recall the economic cap covering theorem, see, e.g., \cite[Lemma 6]{ReitznerCLT2005}. It says that we can find points $y_1,\dotsc,y_m\in \bd K$ and $h_m\geq 0$ such that 
	the caps
	\begin{align*}
		C^K(y_j,h_m) := \{x\in K: (x-y_j) \cdot n_{y_j} \geq -h_m\},
	\end{align*}
	where $n_{y_j}$ is the outer unit normal of $\bd K$ at $y_j$, are pairwise disjoint and satisfy
	\begin{align}
		h_m \sim m^{-\frac{2}{d-1}} &\sim n^{-\frac{2}{d+1}},\\ 
		\Phi(C^K(y_j,h_m)) \leq \Big(\sup_{C^{K}(y_j,h_m)} \varphi\Big) \Vol(C^K(y_j,h_m)) \sim m^{-\frac{d+1}{d-1}} &\sim n^{-1},\label{eqn:upper_bound_C_j}
	\end{align}
	as $n\to\infty$.
	For every cap $C^K(y_j,h_m)$ we consider an approximation of $\bd K$ at $y_j$ with an elliptic paraboloid $Q(y_j)$. Then
	\begin{equation*}
		C^{\tinyonehalf Q(y_j)}(y_j,h_m \delta_0)\subset C^{K}(y_j,h_m\delta_0).
	\end{equation*}
	Furthermore, there are closed sets $D^i(y_j)\subset C^K(y_j,h_m)$, $i\in\{1,\dotsc,d\}$, such that
	\begin{equation}\label{eqn:lower_bound_phi}
		\min\{\Phi(C^{\tinyonehalf Q(y_j)}(y_j,h_m\delta_0)), \Phi(D^1(y_j)), \dotsc,\Phi(D^d(y_j))\} \gtrsim h_m^{\frac{d+1}{2}}\sim n^{-1},
	\end{equation}
	as $n\to\infty$, and such that for all $x_i\in D^i(y_j)$ and a random $Y\in C^{\tinyonehalf Q(y_j)}(y_j,h_m\delta_0)$ distributed according to $\Phi$ relative to $C^{\tinyonehalf Q(y_j)}(y_j,h_m\delta_0)$ we have
	\begin{align}\label{eqn:circular_cone_inside}
		\begin{split}
		\operatorname{pos}_{y_j} [Y,x_1,\dotsc,x_d] &\supset 2Q(y_j)\cap H(n(y_j),y_j-h_m n(y_j)) \\
		&\supset K\cap H(n(y_j),y_j-h_m n(y_j)),
		\end{split}
	\end{align}
	and
	\begin{equation}\label{eqn:var_bound}
		\Var_Y \Psi([Y,x_1,\dotsc,x_d]) \gtrsim h_m^{d+1} \sim n^{-2},
	\end{equation}
	where $\operatorname{pos}_{y_j} [Y,x_1,\dotsc,x_d]=\pos([Y,x_1,\dotsc,x_d]-y_j)+y_j$ stands for the cone with apex $y_j$ spanned by $Y,x_1,\ldots,x_d$ and $n(y_j)$ denotes the outer unit normal vector of $\bd K$ at the point $y_j$.

	\bigskip

	\emph{Step 4 -- estimate the variance by considering special events.}
	For $j\in\{1,\dotsc,n\}$ let $A_j$ be the event that exactly one of the random points $X_1,\ldots,X_n$ is contained in $C^{\tinyonehalf Q(y_j)}(y_j, h_m\delta_0)$ and in each set $D^i(y_j)$, and that no other point is in $C_j=C^K(y_j,h_m\delta_0)$. By \eqref{eqn:upper_bound_C_j} we can find constants $c_1>0$ and $N\in\NN$ such that for all $n\geq N$ we have
	\begin{equation*}
		\Phi(C_j) \leq \frac{c_1}{n}.
	\end{equation*}
	Then, by \eqref{eqn:lower_bound_phi}, we have
	\begin{align*}
		\PP(A_j) &= \binom{n}{d+1} \left(\prod_{i=1}^d\PP(X_i\in D^i(y_j))\right) \PP(X_{d+1} \in C^{\tinyonehalf Q(y_j)}(y_j,h_m\delta_0)) \prod_{\ell=d+2}^m \PP(X_\ell\not\in C_j)\\
		&=\binom{n}{d+1} \Phi(C^{\tinyonehalf Q(y_j)}(y_j,h_m\delta_0)) \prod_{i=1}^d \Phi(D^i(y_j)) (1-\Phi(C_j))^{n-d-1}\\
		&\gtrsim n^{d+1} n^{-d-1} \left(1-\frac{c_1}{n}\right)^{n-d-1}.
	\end{align*}
	Thus there exists $c_2>0$ such that $\PP(A_j)\geq c_2$ for all $n\geq N$ and therefore
	\begin{equation}\label{eqn:sum_A_j}
		\EE \sum_{j=1}^m {\mathbf{1}_{A_j}} = \sum_{j=1}^m \PP(A_j) \geq c_2 m \sim n^{\frac{d-1}{d+1}},\qquad \text{as $n\to \infty$}.
	\end{equation}
	Denote by $\cF$ the $\sigma$-field generated by all random points $X_1,\dotsc,X_n$ except those which
	are contained in caps $C^{\tinyonehalf Q(y_j)}(y_j, h_m\delta_0)$ with $\mathbf{1}_{A_j}=1$.
	Then, by the conditional variance formula,
	\begin{equation*}
		\Var \Psi(K_{\varphi}(n)) = \EE \Var(\Psi(K_{\varphi}(n))|\cF) + \Var \EE(\Psi(K_{\varphi}(n))|\cF)\geq \EE\Var(\Psi(K_{\varphi}(n))|\cF).
	\end{equation*}
	Assume that $\mathbf{1}_{A_j}=\mathbf{1}_{A_k}=1$ for some $j,k\in\{1,\dotsc,m\}$ and further without loss of generality that $X_j$, respectively $X_k$, is the unique point in $C^{\tinyonehalf Q(y_j)}(y_j,h_m\delta_0)$, respectively $C^{\frac{1}{2}Q(y_k)}(y_k,h_m\delta_0)$. By construction the points $X_j$ and $X_k$ are vertices of $K_{\varphi}(n)$, and by
	\eqref{eqn:circular_cone_inside} there is no edge between $X_j$ and $X_k$. Hence the change of $\Psi$-measure of $K_{\varphi}(n)$ if $X_j$ is moved is independent of the change of measure if $X_k$ is moved. This independence yields
	\begin{equation*}
		\Var \Psi(K_{\varphi}(n)|\cF) = \sum_{j=1}^m \mathbf{1}_{A_j} \Var_{X_j} \Psi(K_{\varphi}(n)),
	\end{equation*}
	where the variance under the sum is taken just with respect to the random variable $X_j\in C^{\tinyonehalf Q(y_j)}(y_j,h_m\delta_0)$.
	Combining this with \eqref{eqn:var_bound} and \eqref{eqn:sum_A_j}, implies
	\begin{equation*}
		\Var \Psi(K_{\varphi}(n))
		\gtrsim \frac{1}{n^2} \sum_{j=1}^m \PP(A_j)
		\gtrsim n^{-2+\frac{d-1}{d+1}} = n^{-\frac{d+3}{d+1}},
	\end{equation*}
	as $n\to \infty$. The proof is thus complete.
\end{proof}

\begin{remark}
    We remark that the general idea developed in \cite{ReitznerCLT2005} to prove lower variance bounds for functionals of random polytopes,
    and which is also in the background of the proof of Theorem \ref{thm:LowerVarianceBound},
    has repeatedly and successfully been applied in the literature for various random polytope models.
    We refer the reader to the papers \cite{BaranyFodorWigh, BaranyReitznerVariance, BaranyThaele, BoroczkyFodorReitznerVigh, RichardsonVuWu, TurchiWespi}.
\end{remark}

\subsection{Bounding the first-order difference operators}

In this section we start the actual proof of Theorem \ref{thm:WeightedVolumeEuclidean} by dealing with moments of the first-order difference operator. The proof will be completed in the next section after having dealt with second-order difference operators as well. The next lemma is the main result of this section and we emphasize that the choice for $p$ there is motivated by the application below, where precisely the moments of these orders show up. As above, $c,c_1,c_2$ etc.\ will denote positive and finite constants, which are independent from the parameter $n$.

\begin{lemma}\label{lem:FirstOrderDifferenceOperators}
There exist constants $C,N\in(0,\infty)$ independent from $n$ such that
\begin{align*}
\EE[|D_1f(X)|^p] \leq C^p\Big(\frac{\ln n}{n}\Big)^p
\end{align*}
for all $p\in\{1,2,3,4\}$ and $n\geq N$.
\end{lemma}
\begin{proof}
Recall the definition of the weighted random polytope $K_\varphi(n)=[X_1,\ldots,X_n]$ from \eqref{eq:DefRandomPolytopes}.
Also, taking $\beta=7$ we denote by $c=c_\beta\in(0,\infty)$ the constant implied by Lemma \ref{lem:VanVu}.
Then, we let $A$ be the event that the weighted floating body $K_\delta^\varphi$ with $\delta=c\frac{\ln n}{n}$ is contained in $[X_2,\ldots,X_n]$, i.e.,
\begin{equation*}
    A:=\Big\{K_{c\frac{\ln n}{n}}^\varphi\subset[X_2,\ldots,X_n]\Big\}.
\end{equation*}
From Lemma \ref{lem:VanVu} we conclude that $\PP(A)\geq 1-(n-1)^{-7}\geq 1-c_1 n^{-7}$, where $c_1\in (1,\infty)$ is a suitable constant.

Next, we define for $n\in\NN$, $E_n:=\EE\Psi(K_\varphi(n))$ and put
\begin{equation}\label{eq:DefinitionFuncionF}
    f(X_1,\dotsc,X_k) := \Psi([X_1,\ldots,X_k]) - E_n, \quad \text{for $k\in\{1,\dotsc,n\}$.}
\end{equation}
Note that $f$ is symmetric and Borel measurable, and can be considered as a function $\bigcup_{k=1}^n K^k \to \RR$.

We set $X:=(X_1,\ldots,X_n)$ and, as mentioned in Section \ref{Background}, we define, for $z\in K$ and $\delta\in(0,\infty)$, the visibility region
\begin{equation}\label{eq:Delta}
    \Delta^\varphi(z,\delta) := \{x\in K\setminus \interior K_\delta^\varphi:[z,x]\cap K_\delta^\varphi=\emptyset\}.
\end{equation}
We also use the convention to write $\Delta(z,\delta)$ if $\varphi=\Vol(K)^{-1}$.

Conditioned on the event $A$, we distinguish two cases.
First, if $X_1\in K_\delta^\varphi$ then removing $X_1$ from the sample $X_1,\ldots,X_n$ does not affect the random polytope $K_\varphi(n)$,
which means that in this case $D_1 f(X)=0$.
So, in what follows we can and will restrict ourself to the case that $X_1\in K\setminus K_\delta^\varphi$.
In this situation, that is, if $X\in A$ and $X_1 \in K\setminus K_\delta^\varphi$, we have that
\begin{align*}
    D_1f (X)
    &=\Psi([X_1,\dotsc,X_n]) - \Psi([X_2,\dotsc,X_n])\\
    &\leq \Psi\Big(\Delta^\varphi\Big(X_1,c\frac{\ln n}{n}\Big)\Big)
    \leq \sup_{z\in K\setminus K_{c\frac{\ln n}{n}}^\varphi} \!\!\!\!\!\Psi\Big(\Delta^\varphi\Big(z,c\frac{\ln n}{n}\Big)\Big)
    \leq C_\psi \!\! \sup_{z\in K\setminus K_{c\frac{\ln n}{n}}^\varphi} \!\!\!\!\! \Vol\Big(\Delta^\varphi\Big(z,c\frac{\ln n}{n}\Big)\Big),
\end{align*}
where in the last step we used that
\begin{equation}\label{eq:PsiVSvol}
    \Psi(B)\leq C_\psi\Vol(B) \quad \text{for all Borel sets $B\subset K\setminus K_{c\frac{\ln n}{n}}^\varphi$},
\end{equation}
which holds by assumption on $\psi\in\cW(K)$ if $n$ is large enough.

In a next step, we shall replace the weighted floating body $K_{c \frac{\ln n}{n}}^\varphi$ by the unweighted floating body $K_{c'\frac{\ln n}{n}}$,
which arises by taking the weight function to be equal to $\Vol(K)^{-1}$.
In fact, Lemma~\ref{lem:ComparisonWeightedVSunweighted} implies the existence of constants $c',N\in(0,\infty)$
such that
\begin{equation}\label{eq:WeightedVSUnweighted}
    K_{\frac{\ln n}{c'n}}
    \subset K_{c\frac{\ln n}{n}}^\varphi
    \subset K_{c'\frac{\ln n}{n}}
\end{equation}
holds for all $n\geq N$. Thus, conditioned on $A$ and that $X_1\in K\setminus K_\delta^\varphi$ we have the upper bound
\begin{equation*}
    D_1 f(X)
    \leq C_\psi\sup_{z\in K\setminus K_{\frac{\ln n}{c' n}}}
        \Vol\Big(\Delta\Big(z,\frac{\ln n}{c' n}\Big)\Big)
    \leq C_\psi\sup_{z\in \bd K}
        \Vol\Big(\Delta\Big(z,\frac{\ln n}{c' n}\Big)\Big).
\end{equation*}
By \eqref{eqn:upper_bound_union} we derive the estimate
\begin{equation*}
    D_1 f(X) \lesssim \frac{\ln n}{n},
\end{equation*}
whenever $X\in A$ and $X_1\in K\setminus K_{c\frac{\ln n}{n}}^\varphi$.
Thus, on $A$ we have that
\begin{equation*}
    D_1f(X) \lesssim \frac{\ln n}{n}\,\mathbf{1}\left\{X_1 \not\in K_{c\frac{\ln n}{n}}^\varphi\right\}.
\end{equation*}
On the complementary event $A^C$ we may use the trivial estimate $|D_1f(X)|\leq \Psi(K)=1$. As a consequence we conclude that, for any $p\in\{1,2,3,4\}$ and for large enough $n$,
\begin{align*}
    \EE[|D_1f(X)|^p]
    &= \EE[|D_1f(X)|^p \mathbf{1}_A] + \EE[|D_1f(X)|^p\mathbf{1}_{A^C}]\\
    &\leq c_3^p \Big(\frac{\ln n}{n}\Big)^p\ \PP(A)\ \Psi\Big(K\setminus K_{c\frac{\ln n}{n}}^\varphi \Big) + (1-\PP(A))
    \lesssim \Big(\frac{\ln n}{n}\Big)^p\ \Vol\Big(K\setminus K_{c\frac{\ln n}{n}}^\varphi\Big),
\end{align*}
where $c_3\in(0,\infty)$ is a suitable constant and where we used \eqref{eq:PsiVSvol} again as well as the estimate $\PP(A)\geq 1-c_1n^{-7}$. Applying now Lemma \ref{lem:ComparisonWeightedVSunweighted} and Lemma \ref{lem:VolumeWetPart} we see that
\begin{equation*}
    \EE[|D_1f(X)|^p]
    \lesssim\Big(\frac{\ln n}{n}\Big)^p\Vol\Big(K\setminus K_{\frac{\ln n}{c' n}}\Big)
    \lesssim\Big(\frac{\ln n}{n}\Big)^{p+\frac{2}{d+1}}
\end{equation*}
whenever $n$ is sufficiently large, i.e., if $n\geq N$. This completes the proof.
\end{proof}

The previous bound for moments of the first-order difference operator can now directly applied to the two terms in the normal approximation bound from Lemma \ref{lem:CLTLachiezeReyPeccati} that involve $\gamma_3$ and $\gamma_4$.

\begin{corollary}\label{cor:Gamma3+4}
    There are constants $c,N\in(0,\infty )$ not depending on $n$ such that
    \begin{align*}
        \frac{n}{(\Var f(X))^{3/2}}\ \gamma_3
        &\leq c\ (\ln n)^{3+\frac{2}{d+1}}\ n^{-\frac{1}{2}+\frac{1}{d+1}}.
    \intertext{and}
        \frac{\sqrt{n}}{\Var f(X)}\sqrt{\gamma_4}
        &\leq c\ (\ln n)^{2+\frac{1}{d+1}}\ n^{-\frac{1}{2}+\frac{1}{d+1}},
    \end{align*}
    for all $n\geq N$.
\end{corollary}
\begin{proof}
By definition, $\gamma_3=\EE|D_1f(X)|^3$ and from Lemma \ref{lem:FirstOrderDifferenceOperators} and the lower variance bound in Theorem \ref{thm:LowerVarianceBound} we conclude that
\begin{align*}
    \frac{n}{(\Var f(X))^{3/2}}\ \gamma_3
    &\lesssim \frac{n}{\big(n^{-\frac{d+3}{d+1}}\big)^{3/2}} \Big(\frac{\ln n}{n}\Big)^{3+\frac{2}{d+1}}
    = (\ln n)^{3+\frac{2}{d+1}}\ n^{-\frac{1}{2}+\frac{1}{d+1}}.
\intertext{Similarly,}
    \frac{\sqrt{n}}{\Var f(X)}\sqrt{\gamma_4}
    &\lesssim \frac{\sqrt{n}}{n^{-\frac{d+3}{d+1}}} \sqrt{\Big(\frac{\ln n}{n}\Big)^{4+\frac{2}{d+1}}}
    = (\ln n)^{2+\frac{1}{d+1}}\ n^{-\frac{1}{2}+\frac{1}{d+1}},
\end{align*}
for all large enough $n$. This completes the proof.
\end{proof}

\subsection{Dealing with second-order difference operators}

In this section we will deal with second-order difference operators and complete the proof of Theorem \ref{thm:WeightedVolumeEuclidean}.
In what follows we shall use the same notation as in the previous section, and we let $Y, Y'\!\!, Z, Z'$ be a recombination of our random vector $X=(X_1,\ldots,X_n)$
and let $f$ be given by \eqref{eq:DefinitionFuncionF}.
Again, $c,c_1,c_2$ etc.\ will denote positive and finite constants, which are independent from the parameter $n$.
Our two main estimates read as follows.

\begin{lemma}\label{lem:SecondOrderDifferences}
There are constants $C,N\in(0,\infty)$ independent of $n$ such that
\begin{align*}
    \EE\big[\mathbf{1}\{D_{1,2}f(Y)\neq 0\}\mathbf{1}\{D_{1,3}f(Y')\neq 0\}\,(D_2f(Z))^2(D_3f(Z'))^2\big] &\leq C\Big(\frac{\ln n}{n}\Big)^{6+\frac{2}{d+1}}
\intertext{and}
    \EE\big[\mathbf{1}\{D_{1,2}f(Y)\neq 0\}\,(D_1f(Z))^2(D_2f(Z'))^2\big] &\leq C\Big(\frac{\ln n}{n}\Big)^{5+\frac{2}{d+1}}.
\end{align*}
for all $n\geq N$.
\end{lemma}
\begin{proof}
We denote by $A$ the event
\begin{equation*}
    A:=\Bigg\{K_{c \frac{\ln n}{n}}^\varphi \subset \bigcap_{W\in\{Y,Y'\!\!,Z,Z'\}}[W_4,\dotsc,W_n]\Bigg\},
\end{equation*}
where $c$ is the constant implied by Lemma \ref{lem:VanVu} for the choice $\beta=7$ (note that in the proof of Lemma \ref{lem:FirstOrderDifferenceOperators} we used the symbol $A$ for a different event). Lemma \ref{lem:VanVu} implies that $\PP(A)\geq 1-(n-3)^{-7}\geq 1-c_1 n^{-7}$ for some constant $c_1\in(1,\infty)$. Recall the definition of the visibility region $\Delta^\varphi(z,\delta)$ from \eqref{eq:Delta}. On the event $A$, supposing additionally that $Y_1=Y_1'$, we have the inclusion
\begin{align*}
    &\hspace{-1cm}\{D_{1,2}f(Y)\neq 0,D_{1,3}f(Y')\neq 0\}\\
    &\subset \Big\{\{Y_1,Y_2,Y_3'\}\subset K\setminus K_{c\frac{\ln n}{n}}^\varphi\Big\}\\
    &\qquad\qquad\cap\Big\{\Delta^\varphi\Big(Y_1,c\frac{\ln n}{n}\Big)\cap\Delta^\varphi\Big(Y_2,c\frac{\ln n}{n} \Big)\neq \emptyset\Big\}\\
    &\qquad\qquad\cap\Big\{\Delta^\varphi\Big(Y_1,c\frac{\ln n}{n}\Big)\cap\Delta^\varphi\Big(Y_3',c\frac{\ln n}{n}\Big)\neq \emptyset\Big\}\\
    &\subset\Big\{Y_1\in K\setminus K_{c\frac{\ln n}{n}}^\varphi\Big\}\cap\Bigg\{\{Y_2,Y_3'\}\subset\bigcup_{x\in\Delta^\varphi\big(Y_1,c\frac{\ln n}{n}\big)}\Delta^\varphi\Big(x,c\frac{\ln n}{n}\Big)\Bigg\}.
\end{align*}
Together with the bound for the first-order difference operator provided in Lemma \ref{lem:FirstOrderDifferenceOperators} this implies that
\begin{multline*}
    \EE\big[\mathbf{1}\{D_{1,2}f(Y)\neq 0\}\ \mathbf{1}\{D_{1,3}f(Y')\neq 0\}\ \mathbf{1}\{Y_1=Y_1'\}\ \mathbf{1}_A\,(D_2f(Z))^2\ (D_3f(Z'))^2\big]\\
    \lesssim \PP\Big(Y_1\in K\setminus K_{c\frac{\ln n}{n}}^\varphi\Big)
    \Bigg(\sup_{z\in K\setminus K_{c\frac{\ln n}{n}}^\varphi}
        \Psi\Bigg(\bigcup_{x\in\Delta^\varphi\big(z,c\frac{\ln n}{n}\big)}\Delta^\varphi\Big(x,c\frac{\ln n}{n}\Big)\Bigg)
    \Bigg)^2 \Big(\frac{\ln n}{n}\Big)^4,
\end{multline*}
whenever $n$ is large enough. Now, using \eqref{eq:PsiVSvol}, Lemma \ref{lem:ComparisonWeightedVSunweighted} and Lemma \ref{lem:VolumeWetPart} we see that
\begin{multline*}
    \PP\Big(Y_1\in K\setminus K_{c\frac{\ln n}{n}}^\varphi\Big)
    = \Psi\Big(K\setminus K_{c\frac{\ln n}{n}}^\varphi\Big)
    \lesssim \Vol\Big(K\setminus K_{c\frac{\ln n}{n}}^\varphi\Big)
    \lesssim\Vol\Big(K\setminus K_{\frac{\ln n}{c' n}}\Big)
    \lesssim\Big(\frac{\ln n}{n}\Big)^\frac{2}{d+1}.
\end{multline*}
Similarly, we conclude that
\begin{equation*}
    \sup_{z\in K\setminus K_{c\frac{\ln n}{n}}^\varphi} \!\!\Psi\left(
        \bigcup_{x\in\Delta^\varphi\big(z,c\frac{\ln n}{n}\big)}\!\!\!\Delta^\varphi\Big(x,c\frac{\ln n}{n}\Big)\right)
    \lesssim \sup_{z\in K\setminus K_{\frac{\ln n}{c' n}}} \!\!\Vol\left(
        \bigcup_{x\in\Delta\big(z,\frac{\ln n}{c' n}\big)}\!\!\!\Delta\Big(x,\frac{\ln n}{c' n}\Big)\right)\!.
\end{equation*}
Since by Lemma \ref{cor:est_vis_region} the union in brackets has diameter of order $(\frac{\ln n}{n})^{1/(d+1)}$,
the last expression is bounded by a constant multiple of $\frac{\ln n}{n}$.
Putting things together yields the bound
\begin{multline}\label{eq:Proof2ndOrder1}
    \EE\big[\mathbf{1}\{D_{1,2}f(Y)\neq 0\}\ \mathbf{1}\{D_{1,3}f(Y')\neq 0\}\ \mathbf{1}\{Y_1=Y_1'\}\ \mathbf{1}_A\,(D_2f(Z))^2\ (D_3f(Z'))^2\big]\\
    \lesssim \Big(\frac{\ln n}{n}\Big)^\frac{2}{d+1}\,\Big(\frac{\ln n}{n}\Big)^2\,\Big(\frac{\ln n}{n}\Big)^4= \Big(\frac{\ln n}{n}\Big)^{6+\frac{2}{d+1}},
\end{multline}
provided that $n\geq N$, where $N\in(0,\infty)$ is a suitable constant independent from $n$. We emphasize that taking $Y_1\neq Y_1'$ leads to a term which is of even lower order, because this leads to an extra factor $\Psi(K\setminus K_{c\frac{\ln n}{n}}^\varphi)$, which in turn is of order $(\frac{\ln n}{n})^{2/(d+1)}$. On the event $A^C$ we bound all terms just by $1$ and recall that $\PP(A^C)\lesssim n^{-7}$. This implies that the bound \eqref{eq:Proof2ndOrder1} is valid also without the indicator function $\mathbf{1}_A$. Thus, the proof of the first claim is complete.

In the same way one can treat the second term. In fact, one has the inclusion
\begin{align*}
    \{D_{1,2}f(Y)\neq 0\}
    \subset \Big\{Y_1\in K\setminus K_{c\frac{\ln n}{n}}^\varphi\Big\}
        \cap\Bigg\{Y_2\in\bigcup_{x\in\Delta^\varphi(Y_1,c\frac{\ln n}{n})}\Delta^\varphi\Big(x,c\frac{\ln n}{n}\Big)\Bigg\},
\end{align*}
which together with Lemma \ref{lem:FirstOrderDifferenceOperators} and Lemma \ref{lem:VolumeWetPart} implies the upper bound
\begin{multline*}
    \EE\big[\mathbf{1}\{D_{1,2}f(Y)\neq 0\}\ \mathbf{1}_A\ (D_1f(Z))^2\ (D_2f(Z'))^2\big]\\
    \lesssim \PP\Big(Y_1\in K_{c\frac{\ln n}{n}}^\varphi\Big) \sup_{z\in K\setminus K_{c\frac{\ln n}{n}}^\varphi}
        \Psi\Bigg(\bigcup_{x\in\Delta^\varphi\big(z,c\frac{\ln n}{n}\big)}\Delta^\varphi\Big(x,c\frac{\ln n}{n}\Big)\Bigg) \Big(\frac{\ln n}{n}\Big)^4
    \lesssim \Big(\frac{\ln n}{n}\Big)^{5+\frac{2}{d+1}},
\end{multline*}
if $n$ is large enough. On the event $A^C$ we bound again all terms by $1$ and use the bound for $\PP(A^C)$ to conclude that the previous estimate is still valid without the indicator function $\mathbf{1}_A$. This completes the proof.
\end{proof}

\begin{corollary}\label{cor:Gamma1+2}
There are constants $c,N\in(0,\infty)$ such that
\begin{align*}
    \frac{n^{3/2}}{\Var f(X)}\sqrt{\gamma_1} &\leq c(\ln n)^{3+\frac{1}{d+1}}\ n^{-\frac{1}{2}+\frac{1}{d+1}},
\intertext{and}
    \frac{n}{\Var f(X)}\sqrt{\gamma_2} &\leq c(\ln n)^{{\frac{5}{2}}+\frac{1}{d+1}}\ n^{-\frac{1}{2}+\frac{1}{d+1}}.
\end{align*}
for all $n\geq N$.
\end{corollary}
\begin{proof}
From Lemma \ref{lem:SecondOrderDifferences} and the lower variance bound in Theorem \ref{thm:LowerVarianceBound} we derive that, for sufficiently large $n$,
\begin{align*}
    \frac{n^{3/2}}{\Var f(X)}\sqrt{\gamma_1}
    \lesssim \frac{n^{3/2}}{n^{-\frac{d+3}{d+1}}}\,\sqrt{\Big(\frac{\ln n}{n}\Big)^{6+\frac{2}{d+1}}}
    &= (\ln n)^{3+\frac{1}{d+1}}\ n^{-\frac{1}{2}+\frac{1}{d+1}},
\intertext{and similarly}
    \frac{n}{\Var f(X)}\sqrt{\gamma_2}
    \lesssim \frac{n}{n^{-\frac{d+3}{d+1}}}\sqrt{\Big(\frac{\ln n}{n}\Big)^{5+\frac{2}{d+1}}}
    &= (\ln n)^{{\frac{5}{2}}+\frac{1}{d+1}}\ n^{-\frac{1}{2}+\frac{1}{d+1}}.
\end{align*}
This completes the argument.
\end{proof}

\begin{proof}[Proof of Theorem \ref{thm:WeightedVolumeEuclidean}]
According to Lemma \ref{lem:CLTLachiezeReyPeccati}
\begin{equation*}
    d_{\Wass}\bigg(\cL\Big(\frac{W(n)}{\sqrt{\Var W(n)}}\Big),\cL(Z)\bigg)
    \leq \frac{c\sqrt{n}}{\Var W(n)}\Bigg(\sqrt{n^2\,\gamma_1}+\sqrt{n\,\gamma_2}+\sqrt{\frac{n}{\Var W(n)}}\ \gamma_3+\sqrt{\gamma_4} \Bigg),
\end{equation*}
where $W(n):=f(X)$ is as in \eqref{eq:DefinitionFuncionF}. Notice that $\EE W(n) = 0$ and $0 < \EE W(n)^2 \leq \Psi(K) = 1$.
Using Corollary \ref{cor:Gamma3+4} and Corollary \ref{cor:Gamma1+2} we can find a constant $N\in\NN$ (not depending on $n$) such that
\begin{align*}
    &d_{\Wass}\bigg(\cL\Big(\frac{W(n)}{\sqrt{\Var W(n)}}\Big),\cL(Z)\bigg)\\
    &\qquad\lesssim n^{-\frac{1}{2}+\frac{1}{d+1}}\Big((\ln n)^{3+\frac{1}{d+1}}+(\ln n)^{{\frac{5}{2}}+\frac{1}{d+1}}+(\ln n)^{3+\frac{2}{d+1}}+(\ln n)^{2+\frac{1}{d+1}}\Big)\\
    &\qquad\lesssim (\ln n)^{3+\frac{2}{d+1}}n^{-\frac{1}{2}+\frac{1}{d+1}},
\end{align*}
whenever $n\geq N$. Since the last expression tends to zero, as $n\to\infty$, Theorem \ref{thm:WeightedVolumeEuclidean} follows.
\end{proof}

\section{Proof of Theorems \ref{thm:Sphere}, \ref{thm:Hyperbolic}, \ref{thm:HilbertGeometries}, \ref{thm:dualVolumeExp} and \ref{thm:dualVolume}}\label{sec:ProofNonEuclidean}

It is the purpose of this section to prove Theorems \ref{thm:Sphere}--\ref{thm:dualVolume} from the first section. As we explained after Theorem \ref{thm:WeightedVolumeEuclidean}, they will all follow from the central limit theorem for weighted random polytopes in Euclidean spaces by choosing appropriate weight functions.

\subsection{Central limit theorem in spherical space}

\begin{proof}[Proof of Theorem \ref{thm:Sphere}]
Let $e_1,\ldots,e_{d+1}$ be the standard basis in $\RR^{d+1}$.
Without loss of generality we may assume that $K$ is contained in the open half-sphere $\SS^{d}_+:=\{x\in\RR^{d+1}: x \cdot e_{d+1} >0\}$. Define the gnomonic projection $g:\SS^{d}_+\to\RR^d$ by
\begin{equation*}
    g(x) := \frac{x}{x \cdot e_{d+1}} - e_{d+1},
\end{equation*}
where $\RR^d$ is identified with the hyperplane orthogonal to $e_{d+1}$.
It is well known that the image measure of $\Vol_s$ under $g$ is the Lebesgue measure on $\RR^d$ with the radially symmetric density
$\psi(x)=(1+\|x\|^2)^{-(d+1)/2}$, $x\in\RR^d$, cf.\ \cite[Proposition 4.2]{BesauWernerSpherical}.
Thus, the random variable $\Vol_s(K_s(n))$ from the statement of the theorem has the same distribution as $\Psi(\overbar{K}_{\psi}(n))$,
where $\overbar{K}:=g(K)$ is the image of $K$ under $g$, which implies that $\overbar{K}$ is a (Euclidean) convex body, and $\overbar{K}_\psi(n)$ is the convex hull of $n$ i.i.d.\ random points in $\overbar{K}$ having probability density $\psi/\int_{\overbar{K}}\psi(x)\,\dint x$.

Since the gnomonic projection $g$ is a diffemorphism between the open half-sphere and the Euclidean space we observe that $\bd K$ is a twice differentiable $(d-1)$-submanifold of $\SS^d_+$ if and only if $\overbar{K}$ is a twice differentiable submanifold of $\RR^d$.
Furthermore, the Gauss--Kronecker curvature of $\bd K$ is strictly positive (with respect to the spherical ambient space) if and only if it is strictly positive on $\bd \overbar{K}$ (with respect to the Euclidean ambient space), see \cite[Equation (3.24)]{BesauWernerSpaceForms} for an explicit formula relating the two quantities.
Hence $\overbar{K} \in \cK_+^2(\RR^d)$ if and only if $K\in\cK_+^2(\SS^d_+)$.
Since $\psi\in\cW(\overbar{K})$, we find that $\overbar{K}_{\psi}(n)$ is a weighted random polytope in a Euclidean space $\RR^d$ to which Theorem \ref{thm:WeightedVolumeEuclidean} applies. The claim follows since $\Vol_s(K_s(n)) = \Psi(\overbar{K}_{\psi}(n))$.
\end{proof}

\subsection{Central limit theorem in hyperbolic space}

\begin{proof}[Proof of Theorem \ref{thm:Hyperbolic}]
Let $e_1,\ldots,e_{d+1}\in\RR^{d,1}$ be a Lorentz-orthonormal basis such that $e_{d+1}\in\HH^d$.
We define the gnomonic projection $h:\HH^d\to\RR^d$ by
\begin{equation*}
    h(x) := \frac{x}{x\circ e_{d+1}}+e_{d+1},
\end{equation*}
where $\RR^d$ is again identified with the hyperplane $\{x\in\RR^{d,1}:  x\cdot e_{d+1}=0\}$. It is known that the image measure of $\Vol_h$ under $h$ is the Lebesgue measure on the Euclidean unit ball $B_2^d\subset\RR^d$ with the radially symmetric density $\psi(x)=(1-\|x\|^2)^{-(d+1)/2}$, $x\in \interior B_2^d$, see \cite[Section 3]{BesauWernerSpaceForms}.
As a consequence the random variable $\Vol_h(K_h(n))$ from Theorem \ref{thm:Hyperbolic} has the same distribution as $\Psi(\overbar{K}_\psi(n))$,
where $\overbar{K}:=h(K)$ is again a (Euclidean) convex body and $\overbar{K}_\psi(n)$ is the convex hull of $n$ independent and identically distributed random points on $\overbar{K}$, distributed according to the density $\psi/\int_{\overbar{K}}\psi(x)\,\dint x$.
We again observe that $\overbar{K}\in\cK_+^2(\RR^d)$ since $K\in\cK_+^2(\HH^d)$ and that $\psi\in\cW(\overbar{K})$.
So, as in the spherical case we can apply Theorem \ref{thm:WeightedVolumeEuclidean} and the result follows.
\end{proof}

\begin{remark}\label{rem:projectivemodel}
The gnomonic projection $h$ used in the previous proof is in fact an isometry between the hyperboloid model and the so-called projective model for $\HH^d$ in $B_2^d$.
\end{remark}

\subsection{Central limit theorems in Hilbert Geometries}

\begin{proof}[Proof of Theorem \ref{thm:HilbertGeometries}]
Let $x\in\interior C$ and define for $v\in\RR^d$ the Minkowski norm $\|v\|_x:=\frac{1}{2}(t_+^{-1}+t_-^{-1})$,
where $t_\pm$ are determined by the condition that $x\pm t_\pm v\in\bd C$.
We let $B_{C,x}:=\{v\in\RR^d:\|v\|_x\leq 1\}$ be the associated unit ball (which is known as the harmonic symmetrization of $C$ in $x$)
and denote by $B_{C,x}^\circ:=\{w\in\RR^d: v\cdot w \leq 1\text{ for all }v\in B_{C,x}\}$ the polar of $B_{C,x}$.
The Lebesgue density of the Busemann volume $\Vol_C^{\Bu}$ can then be expressed as $\psi_{\Bu}(x)=\Vol(B_2^d)/\Vol(B_{C,x})$,
$x\in\interior C$.
On the other hand, the Lebesgue density of the Holmes--Thompson volume $\Vol_C^{\HT}$ is given by $\psi_{\HT}(x)=\Vol(B_{C,x}^\circ)/\Vol(B_2^d)$, $x\in\interior C$.
Both functions are continuous and their normalized restrictions $\psi_{\Bu}/\int_K\psi_{\Bu}(x)\,\dint x$ and
$\psi_{\HT}/\int_K\psi_{\HT}(x)\,\dint x$ to $K\in\cK_+^2(C)$ are strictly positive, i.e.,
$\psi_{\Bu},\psi_{\HT}\in\cW(K)$.
This once again puts us into the position to apply Theorem \ref{thm:WeightedVolumeEuclidean}, which then yields the result.
\end{proof}

\begin{remark}\label{rem:HilbertVolumes}
    As already remarked in Section \ref{subsec:CLThilbert} the notion of Holmes--Thompson volume $\Vol^{\HT}_C$
    is closely related to the symplectic structure on $\RR^{2d}$.
    Using the notation introduced in the proof of Theorem \ref{thm:HilbertGeometries} above, we can now explain this connection in more detail.
    At each point $x\in\interior C$ we defined the Minkowski norm $\|\,\cdot\,\|_x$ on the tangent space ${\rm Tan}(\RR^d,x)$ of $\RR^d$ at $x$
    (which in turn was identified with $\RR^d$).
    This induces the Finsler structure $F_C:{\rm Tan}(C)\to[0,\infty), (x,v)=\|v\|_x$ and turns $(\interior C,F_C)$ into a Finsler manifold.
    Here ${\rm Tan}(C)$ is the tangent bundle of $C$ consisting of all pairs $(x,v)$, where $x\in\interior C$ and $v\in {\rm Tan}(\RR^d,x)$ is a tangent vector of $\RR^d$ at $x$.
    Then, if $A\subset\interior C$ is a measurable set and if $\omega$ denotes the standard symplectic form on the cotangent bundle ${\rm CoTan}(C)$ of $\interior C$ we have that
    \begin{equation*}
        \Vol_C^{\HT}(A) = \frac{1}{n!}\int_{A^*}\omega^n,
    \end{equation*}
    where $A^*=\{(x,v)\in {\rm CoTan}(C):x\in A,v\in\BB_{C,x}^\circ\}$.
    We refer to \cite{AlvarezThompson,PapaTroyanov,ThompsonBook} for details and further background material.
\end{remark}

\subsection{Limit theorems for dual volumes}

\begin{proof}[Proof of Theorem \ref{thm:dualVolumeExp}]
    Using polar coordinates and Lutwak's dual Kubota formula \cite[Theorem 1]{Lutwak:1979} we
    may express the dual volume of a convex body $K\in\cK_+^2(\RR^d)$ with $o\in K$ by
    \begin{equation*}
        \overtilde{V}_j(K) = \bbinom{d}{j} \EE\Vol(K\cap E)
        = \frac{1}{\Vol(B_2^{d-j})} \binom{d-1}{j-1} \int_{K} \|x\|^{j-d}\, \dint x,
    \end{equation*}
    for all $1\leq j <d$, see e.g.\ \cite[Lemma 19]{BHK:2019}. Here, $E\in\Gr_j(\RR^d)$ is a random $j$-dimensional linear subspace of $\RR^d$ distributed according to the rotation invariant Haar probability measure.
    Hence, if we set
    \begin{equation}\label{eqn:psi_j}
        \psi_j(x) := \frac{1}{\Vol(B_2^{d-j}) \overtilde{V}_j(K)} \binom{d-1}{j-1} \|x\|^{j-d},
    \end{equation}
    then $\psi_j\in\cW(K)$ (note that $\psi_j(o)=+\infty$ is not a issue, since $\psi_j$ is still integrable).
    The statement of the theorem follows now by applying \cite[Theorem 3.1]{BoroczkyFodorHug2010}, see \eqref{eqn:BFH_Thm}.
\end{proof}

\begin{proof}[Proof of Theorem \ref{thm:dualVolume}]
    Let $\psi_j$ be defined as in \eqref{eqn:psi_j} and apply Theorem \ref{thm:WeightedVolumeEuclidean} with $\psi=\psi_j$ and $\varphi=1/\Vol(K)$.
\end{proof}

\begin{remark}
    Theorem \ref{thm:dualVolumeExp} and \ref{thm:dualVolume} also hold true for the weighted random polytopes $K_{\varphi}(n)$ if $\varphi\in \cW(K)$. That is, we have that
    \begin{equation*}
        \lim_{n\to \infty} n^{\frac{2}{d+1}} \Big(\overtilde{V}_j(K) - \EE \overtilde{V}_j(K_{\varphi}(n))\Big)
        = \tilde{c}(d,j) \int_{\bd K} \varphi(x)^{-\frac{2}{d+1}} H_{d-1}(x)^{\frac{1}{d+1}}\, \|x\|^{j-d} \,\cH^{d-1}(\dint x),
    \end{equation*}
    and
    \begin{equation*}
        \frac{\overtilde{V}_j(K_\varphi(n))-\EE\overtilde{V}_j(K_{\varphi}(n))}{\sqrt{\Var\overtilde{V}_j(K_{\varphi}(n))}}\overset{d}{\longrightarrow}Z,
    \end{equation*}
    as $n\to\infty$.
\end{remark}



\begin{thebibliography}{30}\small

\bibitem{Affentranger:1991}
    F.~Affentranger,
    \emph{The convex hull of random points with spherically symmetric distributions},
    Rend.\ Sem.\ Mat.\ Univ.\ Politec.\ Torino \textbf{49} (1991), 359--383.

\bibitem{AHH:2018}
    D.~Alonso-Guti\'errez, M.~Henk, and M.~A.~Hern\'andez Cifre,
    \emph{A characterization of dual quermassintegrals and the roots of dual Steiner polynomials},
    Adv.\ Math.\ \textbf{331} (2018), 565--588.

\bibitem{AlvarezThompson}
    J.~C.~\'Alvarez Paiva, and A.~Thompson,
    \emph{Volumes on normed and Finsler spaces},
    Math.\ Sci.\ Res.\ Inst.\ Publ.\ \textbf{50} (2004), Cambridge University Press, 1--48.

\bibitem{Barany:1992}
    I.~B\'ar\'any,
    \emph{Random polytopes in smooth convex bodies},
    Mathematika \textbf{39} (1992), 81--92.

\bibitem{BaranySurvey}
    I.~B\'ar\'any,
    \emph{Random polytopes, convex bodies, and approximation},
    Lecture Notes in Math.\ \textbf{1892} (2007), 77--118.

\bibitem{BaranyFodorWigh}
    I.~B\'ar\'any, F.~Fodor, and V.~Vigh,
    \emph{Intrinsic volumes of inscribed random polytopes in smooth convex bodies},
    Adv.\ in Appl.\ Probab.\ \textbf{42} (2010), 605--619.

\bibitem{BaranyHugReitznerSchneider}
    I.~B\'ar\'any, D.~Hug, M.~Reitzner, and R.~Schneider,
    \emph{Random points in halfspheres},
    Random Structures Algorithms \textbf{50} (2017), 3--22.

\bibitem{BaranyReitznerVariance}
    I.~B\'ar\'any, and M.~Reitzner,
    \emph{On the variance of random polytopes},
    Adv.\ Math.\ \textbf{225} (2010), 1986--2001.

\bibitem{BaranyThaele}
    I.~B\'ar\'any, and C.~Th\"ale,
    \emph{Intrinsic volumes and Gaussian polytopes: the missing piece of the jigsaw},
    Documenta Math.\ \textbf{22} (2017), 1323--1335.

\bibitem{Bernig:2014}
    A.~Bernig,
    \emph{The isoperimetrix in the dual Brunn--Minkowski theory},
    Adv.\ Math.\ \textbf{254} (2014), 1--14.

\bibitem{BHPS:2019}
    F.~Besau, T.~Hack, P.~Pivovarov, and F.~E.~Schuster,
    \emph{Spherical centroid bodies},
    preprint (2019), arXiv:\texttt{1902.10614}.

\bibitem{BHK:2019}
    F.~Besau, S.~Hoehner, and G.~Kur,
    \emph{Intrinsic and dual volume deviations of convex bodies and polytopes},
    preprint (2019), arXiv:\texttt{1905.08862}.

\bibitem{BesauLudwigWerner}
    F.~Besau, M.~Ludwig, and E.~M.~Werner,
    \emph{Weighted floating bodies and polytopal approximation},
    Trans.\ Am.\ Math.\ Soc.\ \textbf{370} (2018), 7129--7148.

\bibitem{BesauWernerSpherical}
    F.~Besau, and E.~M.~Werner,
    \emph{The spherical convex floating body},
    Adv.\ Math.\ \textbf{301} (2016), 867--901.

\bibitem{BesauWernerSpaceForms}
    F.~Besau, and E.~M.~Werner,
    \emph{The floating body in real space forms},
    J.\ Differential Geom.\ \textbf{110} (2018), 187--220.

\bibitem{Bogachev}
	V.~I.~Bogachev,
	\emph{Measure Theory, Volume II},
	Springer-Verlag (2007), xiv+575~pp.

\bibitem{BoroczkyHoffmanHug:2008}
    K.~B\"or\"oczky, L.~M.~Hoffmann and D.~Hug,
    \emph{Expectation of intrinsic volumes of random polytopes},
    Period.\ Math.\ Hungar.\ \textbf{57} (2008), 143--164.

\bibitem{BoroczkyFodorHug2010}
    K.~B\"or\"oczky, F.~Fodor, and D.~Hug,
    \emph{The mean width of random polytopes circumscribed around a convex body},
    J.\ London Math.\ Soc.\ \textbf{81} (2010), 499--523.

\bibitem{BoroczkyFodorReitznerVigh}
    K.~B\"or\"oczky, F.~Fodor, M.~Reitzner, and V.~Vigh,
    \emph{Mean width of random polytopes in a reasonably smooth convex body},
    J.\ Multivariate Anal.\ \textbf{100} (2009), 2287--2295.

\bibitem{BHP:2018}
    K.~B\"or\"oczky, M.~Henk, and H.~Pollehn,
    \emph{Subspace concentration of dual curvature measures of symmetric convex bodies},
    J.\ Differential Geom.\ \textbf{109} (2018), 411--429.

\bibitem{BrauchartEtAl}
    J.~S.~Brauchart, A.~B.~Reznikov, E.~B.~Saff, I.~H.~Sloan, Y.~G.~Wang, and R.~S.~Womersley,
    \emph{Random point sets on the sphere-hole radii, covering, and separation},
    Exp.\ Math.\ \textbf{27} (2018), 62--81.

\bibitem{IsotropicConvexBodies}
    S.~Brazitikos, A.~Giannopoulos, P.~Valettas, and B.~H.~Vritsiou,
    \emph{Geometry of Isotropic Convex Bodies},
    Mathematical Surveys and Monographs \textbf{196} (2014), American Mathematical Society, xx+594~pp.

\bibitem{Calkaetc}
    P.~Calka, A.~Chapron, and N.~Enriquez,
    \emph{Mean asymptotics for a Poisson--Voronoi cell on a Riemannian manifold},
    preprint (2018), arXiv:\texttt{1807.09043}.

\bibitem{Chatterjee}
    S.~Chatterjee,
    \emph{A new method of normal approximation},
    Ann.\ Probab.\ \textbf{36} (2008), 1584--1610.

\bibitem{DeussHoerrmannThaele}
    C.~Deu\ss, J.~H\"orrmann, and C.~Th\"ale,
    \emph{A random cell splitting scheme on the sphere},
    Stochastic Process.\ Appl.\ \textbf{127} (2017), 1544--1564.

\bibitem{Gardner:2007}
    R.~J.~Gardner,
    \emph{The dual Brunn-Minkowski theory for bounded Borel sets: dual affine quermassintegrals and inequalities},
    Adv.\ Math.\ \textbf{216} (2007), 358--386.

\bibitem{GJV:2003}
    R.~J.~Gardner, E.~B.~V.~Jensen, A.~Vol\v{c}i\v{c},
    \emph{Geometric tomography and local stereology},
    Adv.\ in Appl.\ Math.\ \textbf{30} (2003), 397--423.

\bibitem{HaberlParpatits:2014}
    C.~Haberl and L.~Parapatits,
    \emph{The centro-affine Hadwiger theorem},
    J.\ Amer.\ Math.\ Soc.\ \textbf{27} (2014), 685--705.

\bibitem{HLYZ:2016}
    Y.~Huang, E.~Lutwak, D.~Yang, and G.~Zhang,
    \emph{Geometric measures in the dual Brunn-Minkowski theory and their associated Minkowski problems},
    Acta Math.\ \textbf{216} (2016), 325--388.

\bibitem{HHRT}
	J.~H\"orrmann, D.~Hug, M.~Reitzner and C.~Th\"ale,
	\emph{Poisson polyhedra in high dimensions},
	Adv.\ Math.\ \textbf{28} (2015), 1--39.

\bibitem{Hug:1996}
    D.~Hug,
    \emph{Contributions to affine surface area},
    Manuscripta Math.\ \textbf{91} (1996), 283--301.

\bibitem{HugReichenbacher}
    D.~Hug, and A.~Reichenbacher,
    \emph{Geometric inequalities, stability results and Kendall's problem in spherical space},
    preprint (2017), arXiv:\texttt{1709.06522}.

\bibitem{HugThaele}
    D.~Hug, and C.~Th\"ale,
    \emph{Splitting tessellations in spherical spaces},
    Electron.\ J.\ Probab.\ \textbf{24} (2019), 60~pp.

\bibitem{KabluchkoMarynychThaeleTemesvari}
    Z.~Kabluchko, A.~Marynych, C.~Th\"ale, and D.~Temesvari,
    \emph{Cones generated by random points on half-spheres and convex hulls of Poisson point processes},
    to appear in Probab.\ Theory Related Fields (2019\textsuperscript{+}), arXiv:\texttt{1801.08008}.

\bibitem{LachPecc}
    R.~Lachi\`eze-Rey, and G.~Peccati,
    \emph{New Berry--Esseen bounds for functionals of binomial point processes},
    Ann.\ Appl.\ Probab.\ \textbf{27} (2017), 1992--2031.

\bibitem{Leichtweiss:1986}
    K.~Leichtweiß,
    \emph{Zur Affinoberfläche konvexer Körper (German) [On the affine surface of convex bodies]},
    Manuscripta Math.\ \textbf{56} (1986), 429--464.

\bibitem{LudwigReitzner:1999}
    M.~Ludwig and M.~Reitzner,
    \emph{A characterization of affine surface area},
    Adv.\ Math.\ \textbf{147} (1999), 138--172.

\bibitem{LudwigReitzner:2010}
    M.~Ludwig and M.~Reitzner,
    \emph{A classification of $\mathrm{SL}(n)$ invariant valuations},
    Ann.\ of Math.\ \textbf{172} (2010), 1219--1267.

\bibitem{Lutwak:1975}
    E.~Lutwak,
    \emph{Dual mixed volumes},
    Pacific J.\ Math.\ \textbf{58} (1975), 531--538.

\bibitem{Lutwak:1979}
    E.~Lutwak,
    \emph{Mean dual and harmonic cross-sectional measures},
    Ann.\ Mat.\ Pura Appl.\ \textbf{119} (1979), 139--148.

\bibitem{Lutwak:1991}
     E.~Lutwak,
     \emph{Extended affine surface area},
     Adv.\ Math.\ \textbf{85} (1991), 39--68.

\bibitem{LYZ:2018}
    E.~Lutwak, D.~Yang, and G.~Zhang,
    \emph{$L_p$ dual curvature measures},
    Adv.\ Math.\ \textbf{329} (2018), 85--132.

\bibitem{MaeharaMartini18}
    H.~Maehara, and H.~Martini,
    \emph{An analogue of Sylvester's four-point problem on the sphere},
    Acta Math.\ Hungar.\ \textbf{155} (2018), 479--488.

\bibitem{PapaTroyanov}
    \emph{Handbook of Hilbert Geometry},
    edited by A.~Papadopoulos and M.~Troyanov,
    IRMA Lect.\ Math.\ Theor.\ Phys.\ \textbf{22} (2014),
    European Mathematical Society, viii+452~pp.

\bibitem{PeccatiReitznerBook}
    \emph{Stochastic Analysis for Poisson Point Processes},
    edited by G.~Peccati and M.~Reitzner,
    Bocconi \& Springer Series \textbf{7} (2016), Springer, xv+346~pp.

\bibitem{PenroseYukichMf}
    M.~D.~Penrose, and Y.~E.~Yukich,
    \emph{Limit theory for point processes in manifolds},
    Ann.\ Appl.\ Probab.\ \textbf{23} (2013), 2161--2211.

\bibitem{Petty:1985}
    C.~M.~Petty,
    \emph{Affine isoperimetric problems},
    Ann.\ New York Acad.\ Sci.\ \textbf{440} (1985), 113--127.

\bibitem{Reitzner:2004}
    M.~Reitzner,
    \emph{Stochastic approximation of smooth convex bodies},
    Mathematika \textbf{51} (2004), 11--29.

\bibitem{ReitznerCLT2005}
    M.~Reitzner,
    \emph{Central limit theorems for random polytopes},
    Probab.\ Theory Relat.\ Fields \textbf{133} (2005), 483--507.

\bibitem{RichardsonVuWu}
    R.~M.~Richardson, V.~H.~Vu, and L.~Wu,
    \emph{An inscribing model for random polytopes},
    Discrete Comput.\ Geom.\ \textbf{39} (2008), 469--499.

\bibitem{SchneiderWeil}
    R.~Schneider, and W.~Weil,
    \emph{Stochastic and Integral Geometry},
    Probability and its Applications (2008), Springer-Verlag, xii+693~pp.

\bibitem{SchuttWerner:1990}
    C.~Sch\"utt and E.~Werner,
    \emph{The convex floating body},
    Math.\ Scand.\ \textbf{66} (1990), 275--290.

\bibitem{SchuttWerner:1994}
    C.~Sch\"utt and E.~Werner,
    \emph{Homothetic floating bodies},
    Geom.\ Dedicata \textbf{49} (1994), 335--348.

\bibitem{SchuttWerner:2003}
    C.~Sch\"utt and E.~Werner,
    \emph{Polytopes with vertices chosen randomly from the boundary of a convex body},
    Lecture Notes in Math.\ \textbf{1807} (2003), 241--422.

\bibitem{Thaele18}
    C.~Th\"ale,
    \emph{Central limit theorem for the volume of random polytopes with vertices on the boundary},
    Discrete Comput.\ Geom.\ \textbf{59} (2018), 990--1000.

\bibitem{ThaeleTurchiWespi}
    C.~Th\"ale, N.~Turchi, and F.~Wespi,
    \emph{Random polytopes: variances and central limit theorems for intrinsic volumes},
    Proc.\ Am.\ Math.\ Soc.\ \textbf{146} (2018), 3063--3071.

\bibitem{ThompsonBook}
    A.~C.~Thompson,
    \emph{Minkowski Geometry},
    Encyclopedia of Mathematics and its Applications \textbf{63} (1996), Cambridge University Press, xvi+346~pp.

\bibitem{Troyanov:2014}
    M.~Troyanov,
    \emph{Funk and Hilbert geometries from the Finslerian viewpoint},
    IRMA Lect.\ Math.\ Theor.\ Phys.\ \textbf{22} (2014), 69--110.

\bibitem{TurchiWespi}
    N.~Turchi, and F.~Wespi,
    \emph{Limit theorems for random polytopes with vertices on convex surfaces},
    Adv.\ in Appl.\ Probab.\ \textbf{50} (2018), 1227--1245.

\bibitem{VuConcentration}
    V.~H.~Vu,
    \emph{Sharp concentration of random polytopes},
    Geom.\ Functional Anal.\ \textbf{15} (2005), 1284--1318.

\bibitem{Werner2002}
    E.~M.~Werner,
    \emph{The $p$-affine surface area and geometric interpretations},
    Rend.\ Circ.\ Mat.\ Palermo (2) Suppl.\ \textbf{70} (2002), 367--382.

\end{thebibliography}
\end{document}